\documentclass[12pt]{amsart}
\usepackage{graphicx}
\usepackage{graphicx}
\usepackage{amsfonts, amsmath, amsthm, amssymb}
\usepackage{xfrac}
\usepackage{cite}
\usepackage{hyperref}
\usepackage{verbatim}
\newtheorem{thm}{Theorem}[section]
\newtheorem{cor}[thm]{Corollary}
\newtheorem{lem}[thm]{Lemma}
\newtheorem{prop}[thm]{Proposition}
\newtheorem{exm}[thm]{Example}
\newtheorem{conj}[thm]{Conjecture}
\newtheorem{defn}[thm]{Definition}
\newtheorem{rem}[thm]{Remark}
\numberwithin{equation}{section}

\newcommand{\norm}[1]{\left\Vert#1\right\Vert}
\newcommand{\abs}[1]{\left\vert#1\right\vert}

\def\eb{{\mathbf{e}}}

\def\bn{{\mathbb N}}

\def\br{{\mathbb R}}

\def\d{\delta}

\def\f{\varphi}

\def\w{\omega}

\def\ab{{\mathbf{a}}}

\def\xb{{\mathbf{x}}}
\def\yb{{\mathbf{y}}}
\def\zb{{\mathbf{z}}}

\def\yb{{\mathbf{y}}}

\begin{document}
    \title[On omega limiting sets]{On omega limiting sets of infinite dimensional Volterra operators}

 \author[F. Mukhemedov]{Farrukh Mukhamedov}
\address{Farrukh Mukhamedov\\
 Department of Mathematical Sciences\\
College of Science, The United Arab Emirates University\\
P.O. Box, 15551, Al Ain\\
Abu Dhabi, UAE} \email{{\tt far75m@gmail.com}, {\tt
farrukh.m@uaeu.ac.ae}}

\author[O. Khakimov]{Otabek Khakimov}
\address{Otabek Khakimov\\
Department of Mathematical Sciences\\
 College of Science, The United Arab Emirates University\\
P.O. Box, 15551, Al-Ain\\
Abu Dhabi, UAE} \email{{\tt hakimovo@mail.ru},
{\tt
otabek.k@uaeu.ac.ae}}

\author[A.F. Embong]{Ahmad Fadillah Embong}
\address{Ahmad Fadillah Embong\\
 Department of Computational \& Theoretical Sciences\\
Faculty of Science, International Islamic University Malaysia\\
P.O. Box, 141, 25710, Kuantan\\
Pahang, Malaysia} \email{{\tt ahmadfadillah.90@gmail.com}}

    \begin{abstract}  In the present paper, we are aiming to study limiting behavior of infinite dimensional Volterra operators. We introduce two classes  $\tilde {\mathcal{V}}^+$ and $\tilde{\mathcal{V}}^-$of infinite dimensional Volterra operators. For operators taken from the introduced classes we  study their omega limiting sets $\omega_V$ and $\omega_V^{(w)}$ with respect to $\ell^1$-norm and pointwise convergence, respectively.  To investigate the relations between these limiting sets, we study linear Lyapunov functions for such kind of Volterra operators. It is proven that if Volterra operator belongs to $\tilde {\mathcal{V}}^+$, then the sets and $\omega_V^{(w)}(\xb)$ coincide for every $\xb\in S$, and moreover, they are non empty. If Volterra operator belongs to $\tilde {\mathcal{V}}^-$, then $\omega_V(\xb)$ could be empty, and it implies the non-ergodicity (w.r.t $\ell^1$-norm) of $V$, while it is weak ergodic.

    \end{abstract}

\maketitle

\section{Introduction} 

It is known \cite{hsbook,11} that nonlinear (in particular, quadratic) mappings appear in various branches
of mathematics and their applications: the theory of differential
equations, probability theory, the theory of dynamical systems, mathematical economics,
mathematical biology, statistical physics, etc.

In the theory of population genetics or game theory, mathematical models are governed by the quadratic differential systems
$$
\frac{d{x}_i}{dt}=\sum_{j,k}a_{jk}^ix_jx_k,\ \ \ i\in\{1,2,\dots,n\},
$$
for a large interacting population of $n$ constituents,
where the numbers $x_i$ represent the fraction of constituents of type $i$, $i\in\{1,2,\dots,n\}$
and satisfy the conservation law $\sum_ix_i=1$ and the tensors $\{a_{jk}^i\}$ of order $n$
having $n^3$ real constants which satisfy
\begin{equation}\label{aij}
a_{jk}^i=a_{kj}^i,\ \ \sum_{i}a_{jk}^i=1,\,\ \  a_{jk}^i\geq 0 \ \ \mbox{for all}\ j\neq i,\ k\neq i, \ i,j,k\in\{1,2,\dots, n\}.
\end{equation}

We notice that the discrete time system corresponding to \eqref{aij} is governed by \textit{quadratic stochastic operators (q.s.o.)}
which appeared in the works of Bernstein \cite{1}.  Such kind of operators
is usually employed to describe the time evolution of species in biology \cite{11}.
It turns out that quadratic dynamical systems are considered an important source of analysis in the study of
dynamical properties and for modeling in various fields, such as population dynamics
\cite{dahl,fish1,hsbook,kest}, physics \cite{plank,udwad}, economics \cite{ulam,ulam1}
and mathematics \cite{hsbook,kest,11}.
Some of the most important findings in the theory of quadratic stochastic operators emerged when
Markov processes were employed to describe some physical and biological systems. Perhaps the best known work on
quadratic models is Lotka-Volterra systems \cite{HHJ,29}.
We note that
the biological treatment of Volterra operators is rather clear: the offspring repeats one of its parents.
In studying the Volterra dynamical systems (when the dynamical system acts on finite dimensional simplex) for a given biological population
one may ask the following question: what kind of genotypes will preserve
and which of them will disappear? A lot of papers are devoted to the investigations of discrete Volterra
operators defined on finite dimensional simpleces \cite{R.gani_tournment,ngrguj,28,zakh}.

In recent decades in the game theory, evolutionary and dynamical
aspects of quadratic dynamical systems have dramatically increased in popularity. Hofbauer and Sigmund's book
\cite{hsbook} serves as a very good introduction to this theory.
We point out that the Volterra operators, in the discrete setting, describe the discrete time zero-sum evolutionary
game dynamics \cite{ngrguj}.
In this direction, zero-sum games and their
evolutionary dynamics were studied by Akin and Losert \cite{AL}.
On the other hand, it is important to investigate dynamics of the Volterra system when the spices in the system is huge \cite{Nag}.
Roughly speaking, what happens if
the game involves a large number of players?
This naturally
leads our attention to the following problem: what is the dynamical behavior of Volterra
operators on an infinite dimensional simplex? In \cite{M2000,Far_Has_Temir} a certain construction of infinite dimensional Volterra operators was studied, but the investigation of their dynamics were left out. In the present paper, will study the limiting behavior of such kind of operators.

We stress that dynamical behavior of Volterra operators on finite dimensional simplex was studied in \cite{R.gani_tournment} by means of
graph tournaments (see also \cite{ngrguj}).  In the mentioned papers,  the compactness of the finite dimensional simplex was essentially used, which allowed to obtain deep results.  
However, research on dynamics for infinite-dimensional dynamical systems is
more complicated and more difficult than that for finite-dimensional ones. In this settings, many things dramatically change and some of well-known facts are violated. In
recent years, there have been attempts to find methods for studying the dynamical and chaotic
behavior of partial differential equations (PDEs) with some promising results \cite{Ch,BFL,SC}.

In the present paper, we first time investigate dynamics of infinite dimensional Volterra operators. Namely, 
we introduce two classes of infinite dimensional Volterra operators, and study their omega limiting sets with respect to $\ell^1$-norm and pointwise convergence. Moreover, we will discuss their relationship, which allows us to investigate ergodic averages of the considered operators.
Our investigations will open a new insight to this topic. 

First recall that a quadratic stochastic operator $V$ is called \textit{Volterra}
if and only if it can be represented
as follows:
\begin{equation}\label{Vol}
(V(\xb))_k=x_k\left(1+\sum_{i=1}^\infty a_{ki}x_i\right),\ \ k\in\mathbb N,
\end{equation}
where $\xb=(x_1,x_2,\dots)\in S$ and
\begin{equation}\label{A1}
a_{ki}=-a_{ik},\ |a_{ki}|\leq1\ \ \mbox{for every }k,i\in\mathbb N.
\end{equation}

By $\mathcal V$ we denote the set of all Volterra operators on infinite dimensional
simplex $S$, and $\mathbb{A}$ denotes the set of all skew-symmetric matrices with \eqref{A1}.
The representation \eqref{Vol} establishes a one-to-one correspondence
$\frak{f}:\mathcal {V}\to\mathbb{A}$ by $\frak{f}(V)=(a_{ki})$. It is clear that $\frak{f}$ is affine, hence
$\mathcal V$ is convex, and moreover, this correspondence allows to investigate certain
geometric properties of  $\mathcal {V}$ by means of structure of the set $\mathbb{A}$ (see \cite{Far_Has_Temir} for more details).

For a given operator $V$
on $S$, by $ \{ V^{n}(\xb_0) \}_{n=1}^{\infty} $  we denote the trajectory of a
point $ \xb_0 \in S $ under $V$. By $ \w_V
(\xb_0) $ (respectively, $ \w_V^{(w)}
(\xb_0) $)  we denote the set of limit points of $ \{ V^{n}(\xb_0) \}_{n=1}^{\infty} $
with respect to $\ell^1$-norm (respectively, pointwise convergence).

In what follows, by a {\it fixed point} of $V$ we mean a vector $\xb\in S$ such that $V(\xb)=\xb$. By
$Fix(V)$ we denote the set of all
fixed points of $V$.

Obviously, if $ \w_V(\xb_0)  $ consists of a single point, i.e. $\omega_V(\xb_0)=\{\xb^*\}$,
then the trajectory  $ \{ V^{n}(\xb_0) \}_{n=1}^{\infty} $ converges to $\xb^*$. Moreover, $\xb^*$ is a fixed
point of $V$. However, looking ahead, we remark
that convergence of trajectories is not a typical case for the
dynamical systems \eqref{Vol}. Therefore, it is of
particular interest to obtain an upper bound for $ \xb_{0} \in S $,
i.e., to determine a sufficiently "small" set containing limiting point $ \xb^*$ under trajectory of Volterra operators.

Denote
\begin{eqnarray*}
&&\mathbb{A}^+=\{(a_{ki})\in\mathbb{A}: \ a_{ki}\geq0\ \mbox{for all }k<i\},\\[2mm]
&&\mathbb{A}^-=\{(a_{ki})\in\mathbb{A}: \ a_{ki}\leq0\ \mbox{for all }k<i\}.
\end{eqnarray*}

We define two subclasses of $\mathcal V$ as follows:
$$
\mathcal V^+=\{V\in\mathcal{V:} \  \frak{f}(V)\in \mathbb{A}^+ \},\ \
\mathcal V^-=\{V\in\mathcal{V}: \  \frak{f}(V)\in \mathbb{A}^- \}.$$

We say a Volterra operator  $V$ belongs to the set $\tilde{\mathcal V}^+$ (resp. $\tilde{\mathcal V}^-$) if there exists some
$k_0\geq1$ such that $a_{ki}=0$ for any $k<k_0<i$ and $a_{ki}\geq0$ (resp. $a_{ki}\leq0$) for every $k_0\leq k<i$.
In other words, $V\in\tilde{\mathcal V}^+$ (resp. $V\in\tilde{\mathcal V}^-$) if and only if either $\frak{f}(V)\in \mathbb{A}^+$
(resp. $\frak{f}(V)\in \mathbb{A}^-$) or there exist an integer $k_0\geq2$ and matrices $A, B$ such that
$$
\frak{f}(V)=\left(
\begin{array}{ll}
A & O\\
O & B
\end{array}
\right)
$$
where $A$ is a $(k_0-1)\times (k_0-1)$-skew-symmetric matrix, $O$ is a null matrix and $B\in\mathbb{A}^+$ (resp. $B\in\mathbb{A}^-$).

We note that $\mathcal V^+\subset\tilde{\mathcal V}^+$, $\mathcal V^-\subset\tilde{\mathcal V}^-$
 and
$\mathcal V^+\cap\mathcal V^-=\{Id\}$, here $Id$ stands for the identity operator.

Let us formulate main results of the present paper.

The next result shows that for operators taken from the class $\tilde{\mathcal V}^+$ their omega limiting sets are not empty and belong to the boundary of the simplex. However, for operators  $V$ taken from the class $\tilde{\mathcal V}^-$ omega limiting set $\omega_{V}({\xb}_0)$ could be empty.

\begin{thm}\label{thm_asos}
Let $V\in \tilde{\mathcal V}^+\cup\tilde{\mathcal V}^-$ and $V\neq Id$. Then the following statements hold:
\begin{enumerate}
\item[$(i)$]
if $V\in\tilde{\mathcal V}^+$ then $\omega_{V}({\xb}_0)=\omega_{V}^{(w)}({\xb}_0)$.
Moreover, $\omega_V^{(w)}(\xb_0)\in\partial S$ for all $\xb_0\in S\setminus{Fix(V)}$\\
\item[$(ii)$] if $V\in\tilde{\mathcal V}^-$ then for any
$\xb_0\in S\setminus{Fix(V)}$ one has:
\begin{enumerate}
\item[$(ii)_a$] $\omega_{V}({\xb}_0)=\omega_{V}^{(w)}({\xb}_0)\in\partial S$ if $\omega_{V}(\xb_0)\neq\emptyset$;\\
\item[$(ii)_b$] if $\omega_{V}({\xb}_0)=\emptyset$ then there exists $r<1$ such that
$\omega_{V}^{(w)}({\xb}_0)\in\partial S_r$.
\end{enumerate}
\end{enumerate}
\end{thm}

From this result, it would be interesting to know about the cardinality of the omega limiting sets.
Next theorem reveals some information the mentioned question.

\begin{thm}\label{thm_asos2}
Let $V\in\tilde{\mathcal V}^+\cup\tilde{\mathcal V}^-$. Then for every
$\xb_0\in S$ the following statements hold:
\begin{enumerate}
\item[$(i)$] if $\omega_V(\xb_0)\neq\emptyset$ then $|\omega_V(\xb_0)|=1\vee\infty$;\\
\item[$(ii)$] $|\omega_{V}^{(w)}({\xb}_0)|=1\vee\infty$.
\end{enumerate}
\end{thm}

It is known \cite{R.gani_tournment} that  any Volterra operator on finite dimensional simplex does not have periodic orbit.  The last theorem implies  that we have a similar kind of situation for  operators from the class $\tilde{\mathcal V}^+\cup\tilde{\mathcal V}^-$.  However, in general, we do not know the structure of the omega limiting sets. That will be a topic for our further investigations.

The next result clarifies location of the omega limiting sets for the considered classes.

\begin{thm}\label{thm_asos3}
Let $V\in\mathcal V^+\cup\mathcal V^-$. Then for any
$\xb_0\in S$ the following statements hold:
\begin{enumerate}
\item[$(i)$] if $V\in\mathcal V^+$ then $\omega_V(\xb_0)=\omega_V^{(w)}(\xb_0)\in\partial S$. Moreover,
$|\omega_V(\xb_0)|=1$;\\
\item[$(ii)$] if $V\in\mathcal V^-$ then one has:
\begin{enumerate}
\item[$(ii)_a$] $\omega_{V}({\xb}_0)=\omega_{V}^{(w)}({\xb}_0)\in\partial S$ if $\omega_{V}(\xb_0)\neq\emptyset$;\\
\item[$(ii)_b$] if $\omega_{V}({\xb}_0)=\emptyset$ then there exists $r<1$ such that
$\omega_{V}^{(w)}({\xb}_0)\in\partial S_r$;\\
\item[$(ii)_c$] $|\omega_V^{(w)}(\xb_0)|=1$.
\end{enumerate}
\end{enumerate}
\end{thm}

In statistical mechanics an ergodic hypothesis proposes a connection
between dynamics and statistics. In the classical theory the assumption
was made that the average time spent in any region of phase space
is proportional to the volume of the region in terms
of the invariant measure, or, more generally, that time averages may be
replaced by space averages.  Therefore, we introduce the following notions.

\begin{defn}
 A q.s.o. $V$ is called
 \begin{enumerate}
  \item[$(i)$] ergodic at $\xb_0\in S$ if the limit
$$
\lim\limits_{n\to\infty}\frac{1}{n}\sum_{n=0}^\infty V^{n}(\xb_0)
$$ exists in $\ell^1$-norm.\\
\item[$(ii)$] weak ergodic at point $\xb_0\in S$ if the limit
$$
\lim\limits_{n\to\infty}\frac{1}{n}\sum_{n=0}^\infty V^{n}(\xb_0)
$$ exists in pointwise convergence.
\end{enumerate}
\end{defn}

On the basis of numerical calculations, Ulam \cite{ulam1} conjectured that an ergodic theorem
holds for any q.s.o. $V$ on finite dimensional simplex. Afterwards, Zakharevich \cite{zakh}
proved that in general this conjecture is false.
In the present paper, we prove that the existence of q.s.o. $V$ on infinite
dimensional simplex for which the ergodic theorem does not hold. Namely, we will prove the following theorem.

\begin{thm}\label{asos4}
Let $V\in\mathcal V^+\cup\mathcal V^-$. Then for any
$\xb_0\in S$ the following statements hold:
\begin{enumerate}
\item[$(i)$] $V$ is weak ergodic at point
$\xb_0$;
\item[$(ii)$] if $V\in\mathcal V^+$ then $V$ is ergodic at point $\xb_0$;\\
\item[$(iii)$] if $V\in\mathcal V^-$ then is ergodic at point $\xb_0$ iff $\omega_V(\xb_0)\neq\emptyset$.
\end{enumerate}
\end{thm}
Thanks to Theorem \ref{asos4}  we infer that  $\omega_V(\xb_0)\neq\emptyset$ then
$V$ is weak ergodic, but not ergodic (w.r.t. $\ell^1$-norm) at that point, while it is weak ergodic.  This is an essential difference between finite and infinite dimensional  settings.
For an explicit example we refer
to Example  \ref{asexample}.

The paper is organized as follows: in section 2, we provide some auxiliary facts on pointwise converges and its relation to $\ell^1$-norm convergence. Moreover, we prove that the unit ball in $\ell^1$ is sequentially weak compact which allows us the further investigation of the limiting set $\omega_V^{(w)}$. Section 3 is devoted to certain properties of Volterra quadratic stochastic operators defined on $S$. In Section 4,  we investigate linear Lyapunov functions for Volterra operators taken from the classes $\mathcal{V}^+$ and $\mathcal{V}^-$.
Furthermore, in Section 5, omega limiting sets of operators taken from the classes $\mathcal{V}^+$ and $\mathcal{V}^-$ are studied. In Section 6 (resp. Section 7) based on results of sections 4 and 5, we investigate dynamics of operators taken from the class $\mathcal{V}^+$ (resp.  $\mathcal{V}^-$).  Finally, Section 8 is devoted to the proof of main results of the present paper.

\section{Pointwise convergence on $\ell^1$}

In this section is devoted to some properties of point-wise convergence in $\ell^1$.

In what follows, as usual,  $\ell^1$  denotes  the space of all absolutely summable sequences with the norm
$\norm{\xb}=\sum\limits_{k=1}^{\infty}\abs{x_k}.$

For a given $r>0$ we denote
$$
{\bf B}_r^+=\{\xb\in\ell^1:\ x_k\geq0\ \mbox{for all } k\in\mathbb N,\ \norm{\xb}\leq r\}
$$
and
$$
S_r=\{\xb\in{\bf B}_r^+: \norm{\xb}=r\}.
$$
In the sequel, the unit sphere $S_1$ is called
{\it  an infinite dimensional simplex}.  Furthermore, for the sake of simplicity, we write $S$ instead of $S_1$.

It is known that $S=convh(Extr S)$, where $Extr(S)$ is the extremal points of $S$ and
$convh(A)$ is the convex hall of a set $A$.
Any extremal point of $S$ has the following form:
\[ \eb_k=(\underbrace{0,\dots,0,1}_{k},,0,\dots ), \ \ k\in \bn\].

Here and henceforth we denote
$$
riS_r = \left\{\xb\in S_r : x_{k}>0,\ k\in\mathbb N\right\}, \ \ \partial S_r = S_r\setminus riS_r.
$$
Let $\{\xb^{(n)}\}$ be a sequence in $\ell_1$. In what follows we write
$\xb^{(n)}\stackrel{\norm{\cdot}}{\longrightarrow}\ab$ instead of
$\norm{\xb^{(n)}-\ab}\rightarrow0$.

\begin{rem}
Note that for any $r>0$ the sets $S_r, B_r$ are  not compact w.r.t. $\ell^1$-norm. In the finite dimensional setting, analogues of these sets are
compact, and hence, the investigation of the dynamics of nonlinear mappings over these kind of sets use well-known methods and techniques of dynamical systems. In our case, the non compactness (w.r.t. $\ell^1$-norm) of the set ${\bf B}_r^+$ complicates our further investigation on dynamics of Volterra operators.
Therefore, we need such a weak topology on  $\ell^1$ so that the set ${\bf B}_r^+$
would be compact with respect to that topology.
\end{rem}

One of weak topologies on $\ell^1$ is the Tychonov topology which generates the pointwise convergence.  We say that
a sequence $\{\xb^{(n)}\}\subset\ell^1$ converges {\it pointwise} to $\xb=(x_1,x_2,\dots)\in\ell^1$ if
$$\lim_{n\to\infty}x^{(n)}_k=x_k\ \ \ \mbox{for every } k\geq1.$$
and write $\xb^{(n)}\stackrel{\mathrm{p.w.}}{\longrightarrow}\xb$.

\begin{rem}\label{conv1}
We notice that the set $\ell^1$ is not closed w.r.t. pointwise topology,
and its completion is $s$ which is the space of all sequences. It is known that this topology is metrizable by the following metric:
\begin{equation}\label{rr}
\rho(\ab,{\bf b})=\sum_{k=1}^\infty2^{-k}\frac{|a_k-b_k|}{1+|a_k-b_k|},\ \ \ \ab,{\bf b}\in s.
\end{equation}
Hence,  for a given sequence $\{\xb^{(n)}\}\subset s$ the following statements are equivalent:
\begin{enumerate}
\item[$(i)$] $\xb^{(n)}\stackrel{\mathrm{p.w.}}{\longrightarrow}\xb$;\\
\item[$(ii)$] $\xb^{(n)}\stackrel{\rho}{\longrightarrow}\xb$.
\end{enumerate}
In the sequel, we will show that the unit ball
of $\ell^1$ is compact w.r.t. pointwise convergence, while
whole $\ell^1$ is not closed in $s$.
\end{rem}

We recall that $\ell^\infty$ is defined to be the space of all bounded sequences endowed with the norm
$$
\norm{\xb}_{\infty }=\sup_{n}\{|x_{n}|\}.
$$
By $c_0$ we, as usual, denote the space of all null sequences,
which is a closed subspace of $\ell^\infty$.

The following lemmas play a crucial role in our further investigations.
\begin{lem}\label{lem1}
Let $\{\xb^{(n)}\}\subset S_r$, for some $r>0$. If $\xb^{(n)}\stackrel{\norm{\cdot}}{\longrightarrow}\ab$, then $\ab\in S_r$.
\end{lem}
\begin{proof}
It is easy to check that $\norm{\xb-\yb}\geq|r-\rho|$, $\forall\xb\in S_r,\ \forall\yb\in S_\rho$.
This fact together with $\xb^{(n)}\stackrel{\norm{\cdot}}{\longrightarrow}\ab$ yields that $\ab\in S_r$.
\end{proof}


\begin{prop}\label{Far}
The set ${\bf B}_1^+$ is sequentially compact w.r.t. the pointwise convergence.
\end{prop}
\begin{proof}
First, we show a sequential compactness of  ${\bf B}_1=\{\xb\in\ell^1: \norm{\xb}\leq1\}$.
Thanks to $c_0^*=\ell^1$, the Alaoglu's Theorem implies that ${\bf B}_1$ is a $\sigma(\ell^1,c_0)$-weak compact.
By $\tau$ we denote the pointwise convergence topology in $\ell^1$. Now define a mapping
$
T:(\ell^1,\sigma)\to(\ell^1,\tau)
$
by $T(x)=x$.
One can check that $T$ is continuous, since $\sigma(\ell^1,c_0)$-weak convergence
implies pointwise convergence.
Hence, $T({\bf B}_1)$ is compact.

On the other hand, the metrizability of $\tau$ (see Remark \ref{conv1}) yields the sequential compactness of ${\bf B}_1$.

Now we show that ${\bf B}_1^+$ is closed w.r.t. pointwise convergence.
Let $\{\xb^{(n)}\}\subset {\bf B}_1^+$ such that $\xb^{(n)}\to\ab$, $\ab=(a_1,a_2,\dots)$.
It is clear that $a_k\geq0$ for all $k\geq1$. This yields that either $\ab\in {
\bf B}_\rho^+$ for some $\rho>0$
or $\ab\not\in\ell_1$.

Suppose that $\ab\not\in{\bf B}_1^+$.
Without lost of generality we may assume that
$$
\sum_{k=1}^{\infty}a_k\geq\rho,\ \ \mbox{for some }\ \rho>1.
$$
Then for any $\varepsilon>0$
there exists an $m\in\bn$ such that
\begin{equation}\label{1}
\sum_{k=1}^{m}a_k>\rho-\frac{\varepsilon}{2}
\end{equation}
On the other hand, from $\lim\limits_{n\to\infty}x^{(n)}_k=a_k,\ k\geq1$, we can find an $n_0\in\bn$ such that
$$
|x^{(n)}_k-a_k|<\frac{\varepsilon}{2m},\ k\in\{1,\dots,m\},\ \  \forall n>n_0.
$$
The last one implies
$$
x_k^{(n)}>a_k-\frac{\varepsilon}{2m},\ k\in\{1,\dots,m\},\  \ \  \textrm{for all} \ \  n>n_0.
$$
Hence,
\begin{eqnarray*}
\sum_{k=1}^mx_k^{(n)}&>&\sum_{k=1}^m\left(a_k-\frac{\varepsilon}{2m}\right)\\[2mm]
&=&\sum_{k=1}^ma_k-\frac{\varepsilon}{2}, \ \  \textrm{for all} \ \  n>n_0.
\end{eqnarray*}
So,  the last inequality  together with  \eqref{1} implies
\begin{equation}\label{eq2}
\sum_{k=1}^mx_k^{(n)}>\rho-\varepsilon,\ \ \forall n>n_0.
\end{equation}
We know that for all $n\geq1$ one has $\sum_{k=1}^\infty x_k^{(n)}\leq1$ which together with \eqref{eq2}
implies $\varepsilon>\rho-1$. This contradicts to the arbitrariness of $\varepsilon$.
So, we conclude that $\ab\in{\bf B}_1^+$.

Consequently, as a closed subset of sequential compact set ${\bf B}_1$,
the set ${\bf B}_1^+$ is also sequentially compact. This completes the proof.
\end{proof}

It is clear that $\xb^{(n)}\stackrel{\norm{\cdot}}{\longrightarrow}\ab$ implies
$\xb^{(n)}\stackrel{\mathrm{p.w.}}{\longrightarrow}\ab$.
A natural question arises: is there any equivalence criteria for these two types of convergence on some set?
Next result gives a positive
answer to this question.
\begin{lem}\label{lem4}
Let $\{\xb^{(n)}\}$ be a sequence on $S_r$. Then the following statements are equivalent:
\begin{enumerate}
\item[(1)] $\xb^{(n)}\stackrel{\norm{\cdot}}{\longrightarrow}\ab$ and $\ab\in S_r$;\\
\item[(2)] $\xb^{(n)}\stackrel{\mathrm{p.w.}}{\longrightarrow}\ab$ and $\ab\in S_r$.
 \end{enumerate}
\end{lem}
\begin{proof} It is enough to prove the implication $(2)\Rightarrow(1)$. Let
$\xb^{(n)}\stackrel{\mathrm{p.w.}}{\longrightarrow}\ab$ and $\ab\in S_r$. Pick any positive number $\varepsilon$.
Since $\norm{\ab}=r$, there exists an integer $m\geq1$ such that
\begin{equation}\label{l4eq1}
\sum_{k=1}^ma_k>r-\frac{\varepsilon}{4}.
\end{equation}
The convergence $\xb^{(n)}\stackrel{\mathrm{p.w.}}{\longrightarrow}\ab$ implies the existence of
an integer $n_0\geq1$
such that
\begin{equation}\label{l4eq2}
\left|x^{(n)}_k-a_k\right|<\frac{\varepsilon}{4m},\ \ k\in\{1,\dots.m\}\ \ \ \ \forall n>n_0.
\end{equation}
From \eqref{l4eq2} using \eqref{l4eq1} we obtain
\begin{equation}\label{yulduzcha}
\sum_{k=1}^mx^{(n)}_k>\sum_{k=1}^m\left(a_k-\frac{\varepsilon}{4m}\right)>r-\frac{\varepsilon}{2},\ \ \forall n>n_0.
\end{equation}
Due to $\norm{\xb^{(n)}}=r$ for any $n\geq1$, from \eqref{yulduzcha} one gets
\begin{equation}\label{l4eq3}
\sum_{k=m+1}^\infty x^{(n)}_k<\frac{\varepsilon}{2},\ \ \ \forall n>n_0.
\end{equation}
Hence, using \eqref{l4eq1}-\eqref{l4eq3} we have
\begin{eqnarray*}
\norm{\xb^{(n)}-\ab}&=&\sum_{k\leq m}\left|x^{(n)}_k-a_k\right|+\sum_{k>m}\left|x^{(n)}_k-a_k\right|\\
&\leq&\sum_{k\leq m}\left|x^{(n)}_k-a_k\right|+\sum_{k>m}\left|x^{(n)}_k\right|+\sum_{k>m}\left|a_k\right|\\
&<&\frac{\varepsilon}{4}+\frac{\varepsilon}{2}+\frac{\varepsilon}{4}=\varepsilon,\ \ \ \ \ \forall n>n_0,
\end{eqnarray*}
which means that $\xb^{(n)}\stackrel{\norm{\cdot}}{\longrightarrow}\ab$. This completes the proof.
\end{proof}

Recall that a
functional $\varphi:\ell^1\to\mathbb R$ is called {\it pointwise continuous} if
for any $\ab\in\ell^1$ and any sequence $\{\xb^{(n)}\}\subset\ell^1$ with $\xb^{(n)}\stackrel{\mathrm{p.w.}}{\longrightarrow}\ab$
one has
$\varphi(\xb^{(n)})\rightarrow\varphi(\ab)$.

Now we provide a criteria for linear functionals to be pointwise continuous.

Given ${\bf b}\in\ell^\infty$, let us define
\begin{equation}\label{uffff}
\varphi_{\bf b}(\xb)=\sum_{k=1}^\infty b_kx_k,\ \ \ \ \xb\in\ell^1.
\end{equation}
\begin{lem}\label{lem5}
Let ${\bf b}\in\ell^\infty$, then the linear functional $\varphi_{\bf b}$
is pointwise continuous on ${\bf B}_1^+$ iff ${\bf b}\in c_0$.
\end{lem}
\begin{proof}
Assume that $\varphi_{\bf b}$ is a pointwise continuous. Consider the sequence
$\{\eb_n\}$ for which one has $\eb_n\stackrel{\mathrm{p.w.}}{\longrightarrow}\bf{0}$,
where ${\bf{0}}=(0,0,\dots)$.
From $\varphi_{\bf b}(\eb_n)=b_n$, $\varphi_{\bf b}({\bf{0}})=0$ and the pointwise continuity of
$\varphi_{\bf b}$ implies $b_n\to0$ as $n\to\infty$.

Now let us suppose that $b_k\to0$ as $k\to\infty$, and take any sequence $\{\xb^{(n)}\}\subset {\bf B}_r^+$ such that
$\xb^{(n)}\stackrel{\mathrm{p.w.}}{\longrightarrow}\xb$.
We will show that $\varphi_{\bf b}(\xb^{(n)})\to\varphi_{\bf b}(\xb)$.
If $\norm{{\bf b}}_\infty=0$ then nothing to proof.
So, we consider $\norm{{\bf b}}_\infty\neq0$.

Take an arbitrary positive number $\varepsilon$. Then there exists an integer $m\geq1$
such that $|b_k|<\frac{\varepsilon}{4r}$ for all $k>m$.
The pointwise convergence $\xb^{(n)}\stackrel{\mathrm{p.w.}}{\longrightarrow}\xb$
 implies the existence of an integer $n_0$ such that
$$
|x^{(n)}_k-x_k|<\frac{\varepsilon}{2\norm{{\bf b}}_\infty},\ \ k\in\{1,\dots,m\},\ \ \ \forall n>n_0.
$$
Consequently, we have
\begin{eqnarray*}
\left|\varphi_{\bf b}(\xb^{(n)})-\varphi_{\bf b}(\xb)\right|&\leq&\left|\sum_{k\leq m}b_k(x_k^{(n)}-x_k)\right|+
\left|\sum_{k>m}b_k(x_k^{(n)}-x_k)\right|\\
&\leq&\sum_{k\leq m}\left|b_k(x_k^{(n)}-x_k)\right|+\sum_{k>m}\left|b_k(x_k^{(n)}-x_k)\right|\\
&\leq&\norm{{\bf b}}_\infty\sum_{k\leq m}\left|x_k^{(n)}-x_k\right|+\frac{\varepsilon}{4r}\sum_{k>m}\left|x_k^{(n)}-x_k\right|\\
&<&\norm{{\bf b}}_\infty\cdot\frac{\varepsilon}{2\norm{{\bf b}}}_\infty+\frac{\varepsilon}{4r}\cdot 2r\\[2mm]
&=&\varepsilon,\ \ \  \textrm{for all} \ \  n>n_0.
\end{eqnarray*}
This yields the desired assertion.
\end{proof}

\section{Volterra Quadratic Stochastic Operators}

Let
$$
S^{d-1}=\left\{\xb=(x_1,x_2,\dots,x_d)\in\mathbb R^d: \sum_{i=1}^dx_i=1,\ x_k\geq0,\ k\in\{1,\dots,d\}\right\}
$$
which is the $(d-1)$-dimensional simplex. Recall that
a {\it quadratic stochastic operator}
$V$ on $S^{d-1}$ is a mapping defined by
\begin{eqnarray}\label{eqn_qso_fd}
	(V(\xb))_{k} = \sum\limits_{i,j=1}^{d}p_{ij,k}x_ix_j\quad , k\in\{1,2,\dots,d\}
\end{eqnarray}
where
\begin{eqnarray}\label{p_ijk}
	p_{ij,k}\geq0, \quad p_{ij,k}=p_{ji,k}, \quad \sum\limits_{k=1}^{d}p_{ij,k} = 1, \quad i,j,k\in\{1,2,\dots,d\}.
\end{eqnarray}

Quadratic stochastic operators were first introduced by Bernstein \cite{1}.
Such operators frequently arise in many models of
mathematical genetics, namely, the theory of heredity \cite{AL,hsbook,11}. To the investigation of quadratic stochastic operators
it was devoted many papers (see for example, \cite{R.gani_tournment,ngrguj,6,MG2015,fish1,kest,20,25,28,29}).

Recall that the operator \eqref{eqn_qso_fd},\eqref{p_ijk} is called {\it Volterra}, if $p_{ij,k}=0$ for any $k\notin\{i,j\}$.
The biological
treatment of such operators is rather clear: the offspring repeats one of its parents.

The following result has been proved in \cite{R.gani_tournment}.

\begin{thm}\label{rg1}\cite{R.gani_tournment}
Let $V$ be a Volterra q.s.o. on $S^{d-1}$. Then the following
statements hold:
\begin{enumerate}
\item[$(i)$] if $\xb_0\in S^{d-1}\setminus Fix(V)$ then $\omega_V(\xb_0)\in\partial S^{d-1}$;
\\
\item[$(ii)$] if $\xb_0\in S^{d-1}$ then $|\omega_V(\xb_0)|=1\vee\infty$.
\end{enumerate}
\end{thm}

\begin{rem}
We stress that due to the finite dimensionality of the simplex
$S^{d-1}$, $d\geq2$, for any q.s.o. $V$ on $S^{d-1}$, $d\geq2$ one has
$\omega_V(\xb_0)=\omega_V^{(w)}(\xb_0)$ for every $\xb_0\in S^{d-1}$. When one
considers an infinite dimensional setting, then the indicated
equality may be violated. Consequently, the investigation
the dynamics of infinite dimensional Volterra operators
become a tricky job.
\end{rem}

In the present paper,
our main aim is to study the dynamics of infinite dimensional
Volterra q.s.o.

Let $V$ be a mapping on the infinite dimensional simplex $S$ defined by
\begin{eqnarray}\label{eqn_qso}
(V(\xb))_{k} = \sum\limits_{i,j=1}^{\infty}p_{ij,k}x_ix_j,\quad k = 1,2,3,\dots
\end{eqnarray}
Here, $\{p_{ij,k}\}$ are the hereditary coefficients which satisfy
\begin{eqnarray}
p_{ij,k}\geq0, \quad p_{ij,k}=p_{ji,k}, \quad \sum\limits_{k=1}^{\infty}p_{ij,k} = 1, \quad i,j,k=1,2,3,\dots.
\end{eqnarray}
It is important to notice that the mapping $V$ is well-defined i.e., $V(S)\subset S$.
Such kind of mapping $V$ is called
 {\it quadratic stochastic operator} (q.s.o.).

By \textit{support} of $\xb = (x_{1},\dots, x_{n}, \dots) $ we mean a set $ supp(\xb) = \left\{ i
\in \bn:\ x_{i} \neq 0 \right\}$.

Likewise as a finite dimensional case a q.s.o. $ V : S \rightarrow S $ is  called {\it Volterra} if
\begin{eqnarray}\label{eqn_cond_Vol}
p_{ij,k} =0 \textmd{ if } k \notin \{i,j\},\ i,j=1,2,3,\dots
\end{eqnarray}

Taking into account \eqref{eqn_qso}, one easily can check that
\eqref{eqn_cond_Vol} is equivalent to \eqref{Vol}.
In \cite{Far_Has_Temir} several properties of infinite dimensional
Volterra operators have been investigated.

In what follows, by $\mathcal V$ we denote the set of all Volterra q.s.o. defined on $S$.
Note that $\mathcal V$
is a convex set.

For any $I\subset\mathbb N$ we define a subset of $S$ as follows
$$
\Gamma_I=\{\xb\in S: x_i=0, i\in I\}.
$$
The subset $\Gamma_I$ is called a {\it face of the simplex}.

Let us recall some known facts for Volterra q.s.o. on $S$.

\begin{prop}\cite{Far_Has_Temir}\label{5vo5}  Let $V\in\mathcal V$. Then
the following assertions hold:
\begin{enumerate}
\item[$(i)$] for every $I\subset\mathbb N$ one has $V(\Gamma_I)\subset\Gamma_I$,
$V(ri\Gamma_I)\subset ri\Gamma_I$;
\item[$(ii)$] $ Extr(S)\subset Fix(V)$;
\item[$(iii)$] $ V(riS) \subset riS $;
\end{enumerate}
\end{prop}

Here and henceforth, we use $ V^{n}(\xb) $ to denote the
iterations of the given q.s.o. $ V $ at the initial point $ \xb\in S $ i.e.,
$$
V^{n+1}(\xb) = V(V^{n}(\xb)), \ n \in \bn.
$$

Recall that the $\omega$-limit set of a point $\xb_0\in S$  w.r.t. $V$ is
$$
\w_{V}(\xb_0):=\bigcap_{n\geq 1}\overline{\bigcup_{k\geq n}V^k(\xb_0)}^{\|\cdot\|}.
$$
Equivalently, $ \xb^{*} \in \w_{V}(\xb_0) $ means that
there exists a subsequence $ \{ n_{k} \} $ such that
$$
V^{n_{k}}(\xb_0)\stackrel{\norm{\cdot}}{\longrightarrow}\xb^{*},\ \ \ n_k\to\infty.
$$

Now define weak $\omega$-limit set of a point $\xb_0\in S$  w.r.t. $V$ by
$$
\w_{V}^{(w)}(\xb_0):=\bigcap_{n\geq 1}\overline{\bigcup_{k\geq n}V^k(\xb_0)}^{\rho},
$$
here $\rho$ is the metric given by \eqref{rr}.

So, $ \xb^{*} \in \w_{V}^{(w)}(\xb_0) $ means that
there exists a subsequence $ \{ n_{k} \} $ such that
$$
V^{n_{k}}(\xb_0)\stackrel{\mathrm{p.w.}}{\longrightarrow}\xb^{*},\ \ \ n_k\to\infty.
$$

\begin{rem}
The compactness of $S^{d-1}$ implies $\w_V(\xb_0)\neq\emptyset$ for any $\xb_0\in S^{d-1}$.
It turns out that this property is violated in the infinite dimensional setting
 (see Example \ref{asexample}).
If
we consider q.s.o. $V$ on infinite dimensional simplex, then
according to Lemma \ref{Far} one has $\w_V^{(w)}(\xb_0)\neq\emptyset$ for
any $\xb_0\in S$. That
fact gives a motivation in studying relationship between the sets $\w_V(\xb_0)$ and $\w_V^{(w)}(\xb_0)$
\end{rem}

\begin{exm}\label{asexample} Let $(a_{ij})_{i,j\geq1}$ be an infinite dimensional skew-symmetric matrix
such that $a_{ki}=-1$ for all $i>k$. Then the corresponding Volterra operator $V$ belongs to $\mathcal V^-$
and has the following form
$$
(V(\xb))_k=\left\{
\begin{array}{ll}
x_1^2, & \mbox{if }\ k=1,\\
x_k^2+2x_k\sum_{i=1}^{k-1}x_i, & \mbox{if }\ k\geq2.
\end{array}
\right.\ \ \ \ \xb\in S.
$$
From the last expressions, for any $m,n\in\bn$, one gets
\begin{equation}\label{asexameq}
\sum_{k=1}^m(V^{n}(\xb))_k=\left(\sum_{k=1}^mx_k\right)^{2^n}.
\end{equation}
Let us assume that $|supp(\xb)|=\infty$. Then for any $m\geq1$ we have
$\sum_{k=1}^mx_k<1$. Hence,
for any fixed $m\geq1$ from \eqref{asexameq} one has $\sum_{k=1}^mV^{n}(\xb)_k\to0$ as $n\to\infty$.
Consequently, $V^{n}(\xb)\stackrel{\mathrm{p.w.}}{\longrightarrow}{\bf0}$.
Due to ${\bf0}\not\in S$
and Lemma \ref{lem4} one concludes
$\w_V(\xb)=\emptyset$.

Now let us suppose that $|supp(\xb)|=m_0$. Then we have $\sum_{k=1}^{m_0}x_k=1$
and $x_k=0$ for any $k>m_0$. So, from Proposition \ref{5vo5} one has
$(V^n(\xb))_k=0$ for any $n\in\bn$ and $k>m_0$.

On the other hand, from $\sum_{k=1}^{m_0-1}x_k<1$ it follows that
$$
\sum_{k=1}^{m_0}(V^n(\xb))_k=\left(\sum_{k=1}^{m_0}x_k\right)^{2^n}\to0,\ \ \mbox{ as } n\to\infty
$$
It yields
that
\begin{equation}\label{dgfg}
(V^n(\xb))_k\to0,\ \ \mbox{for any } k<m_0
\end{equation}
Finally, the equality
$$
\sum_{k=1}^{m_0}(V^n(\xb))_k=\left(\sum_{k=1}^{m_0}x_k\right)^{2^n}=1
$$
together with \eqref{dgfg} implies $(V^n(\xb))_{m_0}\to1$, which
means that $V^{n}(\xb)\stackrel{\mathrm{p.w.}}{\longrightarrow}{\bf e}_{m_0}$. Hence, by Lemma
\ref{lem4} we obtain $V^{n}(\xb)\stackrel{\norm{\cdot}}{\longrightarrow}{\bf e}_{m_0}$.

Consequently, for any $\xb\in S$ we find that
\begin{equation}\label{mis1}
\w_V(\xb)=\left\{
\begin{array}{ll}
\eb_{m_0}, & \mbox{if }\ \max\{supp(\xb)\}=m_0,\\[2mm]
\emptyset, & \mbox{if }\ |supp(\xb)|=\infty.
\end{array}
\right.
\end{equation}
\begin{equation}\label{mis2}
\w^{(w)}_V(\xb)=\left\{
\begin{array}{ll}
\eb_{m_0}, & \mbox{if }\ \max\{supp(\xb)\}=m_0,\\[2mm]
{\bf0}, & \mbox{if }\ |supp(\xb)|=\infty.
\end{array}
\right.
\end{equation}
From \eqref{mis1} one can see that $V$ is ergodic at $\xb_0\in S$ if $|supp(\xb_0)|<\infty$.

Now we consider a case
$|supp(\xb_0)|=\infty$ and we show that $V$ is not ergodic at $\xb_0$.

Assume that $V$ is ergodic at $\xb_0\in S$ ($|supp(\xb_0)|=\infty$). Then there exists
$\hat{\xb}\in S$ such that
\begin{equation}\label{mis3}
\frac{1}{n}\sum_{k=0}^nV^k(\xb_0)\stackrel{\norm{\cdot}}{\longrightarrow}\hat{\xb},\ \ \mbox{as }n\to\infty.
\end{equation}
Then by Lemma \ref{lem4} we obtain
$$
\frac{1}{n}\sum_{k=0}^nV^k(\xb_0)\stackrel{\mathrm{p.w.}}{\longrightarrow}\hat{\xb},\ \ \mbox{as }n\to\infty.
$$
On the other hand, from \eqref{mis2} it follows that
\begin{equation}\label{mis4}
\frac{1}{n}\sum_{k=0}^nV^k(\xb_0)\stackrel{\mathrm{p.w.}}{\longrightarrow}{\bf 0},\ \ \mbox{as }n\to\infty.
\end{equation}

Hence, \eqref{mis3},\eqref{mis4} implies $\hat{\xb}=\bf0$, which contradicts to $\hat{\xb}\in S$.
So, we infer that $V$ is not ergodic at point $\xb_0\in S$.
\end{exm}

\section{Lyapunov functions for Volterra q.s.o.}

In this section, we construct two types of Lyapunov functions for Volterra q.s.o. with respect to
classes $\mathcal V^+,\mathcal V^-$.

\begin{defn} A $\ell^1$-continuous function $ \varphi: S \rightarrow \br $ is called a
\textit{Lyapunov function} for q.s.o. $V$ if the limit $ \lim\limits_{n
	\rightarrow \infty } \varphi(V^{n}(\xb))$ exists for any initial
point $ \xb\in S$.
\end{defn}
Obviously, if $\varphi$ is Lyapunov function for q.s.o. $V$ and
$\lim\limits_{n \rightarrow \infty } \varphi(V^{n}(\xb_{0}))=\xb^*$,
then $ \w_V (\xb_{0}) \subset \varphi^{-1}(\xb^*)$. Consequently, to determine more
precisely of $\w_V (\xb_{0})$ we should construct as much as possible
Lyapunov functions.

\begin{thm}\label{thm_Ly_da<0}
Let $V\in\mathcal V$ and $\frak{f}(V)=(a_{ij})$ be its corresponding skew-symmetric matrix. Assume that
$ {\bf b}\in \ell^{\infty} $ such that for any pair $(k,i)\in\bn^2$ one has
$ b_{k}a_{ki}\leq 0 $ (resp. $ b_{k}a_{ki}\geq 0 $). Then the functional $\f_{\bf b}$ given by \eqref{uffff} on
$S$
is a Lyapunov function for $V$.
\end{thm}
\begin{proof} It is easy to see that the functional
$\varphi_{\bf b}$ is well-defined on $S$ (even on $\ell^1$).
One can check
\begin{eqnarray}\label{starchik}
\f_{{\bf b}}(V(\xb))& =& \sum\limits_{k=1}^{\infty}b_{k}x_{k} \left( 1+\sum\limits_{i=1}^{\infty}a_{ki}x_{i} \right)\nonumber\\[2mm]
&=&\sum\limits_{k=1}^{\infty}\left( b_{k}x_{k} + b_{k}x_{k}\sum\limits_{i=1}^{\infty}a_{ki}x_{i} \right)
\end{eqnarray}
Due to
	\begin{eqnarray}
		\left\vert \sum\limits_{k=1}^{\infty}  b_{k}x_{k} \sum\limits_{i=1}^{\infty} a_{ki}x_{i} \right\vert &\leq&
		\left\vert  b_{1}x_{1}\sum\limits_{i=1}^{\infty}a_{1i}x_{i} \right\vert +
\left\vert b_{2}x_{2}\sum\limits_{i=1}^{\infty}a_{2i}x_{i} \right\vert + \cdots \nonumber \\
		&\leq& |b_{1}|x_{1}\sum\limits_{i=1}^{\infty}|a_{1i}|x_{i} + |b_{2}|x_{2}\sum\limits_{i=1}^{\infty}|a_{2i}|x_{i} + \cdots \nonumber \\
		&\leq& |b_{1}|x_{1} \sum\limits_{i=1}^{\infty}x_{i} + |b_{2}|x_{2} \sum\limits_{i=1}^{\infty}x_{i} + \cdots \nonumber \\
		&\leq& \sum\limits_{k=1}^{\infty} |b_{k}|x_{k} \nonumber \\
		&\leq& \norm{\bf b}_\infty \sum\limits_{k=1}^{\infty} x_{k} \nonumber \\
		& = & \norm{\bf b}_\infty \nonumber,
	\end{eqnarray}
we infer that the series $ \sum\limits_{k=1}^{\infty}  b_{k}x_{k} \sum\limits_{i=1}^{\infty} a_{ki}x_{i} $ converges.
	Therefore, from \eqref{starchik} we obtain
	\begin{eqnarray}
		\f_{{\bf b}}(V(\xb)) & =& \sum\limits_{k=1}^{\infty} b_{k}x_{k} +
\sum\limits_{k=1}^{\infty} b_{k}x_{k}\sum\limits_{i=1}^{\infty}a_{ki}x_{i}  \nonumber \\
		&=& \f_{{\bf b}}(\xb) + \sum\limits_{k=1}^{\infty} b_{k}x_{k}\sum\limits_{i=1}^{\infty}a_{ki}x_{i}  \nonumber
	\end{eqnarray}
The assumption of Theorem yields that $ x_{k}\sum\limits_{i=1}^{\infty}b_{k}a_{ki}x_{i} \leq 0 $ for every $ k\in \bn $.
	Hence,
$$
\f_{{\bf b}}(V(\xb)) \leq \f_{{\bf b}}(\xb)
$$
which implies that $ \lim\limits_{n \rightarrow \infty}\f_{{\bf b}}(V^{n}(\xb))  $ converges.
	This completes the proof.
\end{proof}
We notice that the set of Volterra q.s.o. that satisfies the condition of Theorem \ref{thm_Ly_da<0} is non-empty. Consider the following example
\begin{exm}
	Let us choose the following skew-symmetric matrix:
	 \[
	 \left[
	 \begin{array}{cccccccccccc} 
	 0 & a_{12} & a_{13} & \cdots & a_{1m-1} & -1 & a_{1m+1} & \cdots & a_{1n-1} & 1 & a_{1n+1} & \cdots \\
	 a_{2,1} & 0 & a_{23} & \cdots & a_{2m-1} & -1 & a_{2m+1} & \cdots & a_{2n-1} & 1 & a_{2n+1} & \cdots \\
	 \vdots& \vdots&\vdots & \vdots& \vdots& \vdots& \vdots& \vdots& \vdots& \vdots& \vdots& \vdots \\
	 a_{m-11} & a_{m-12} & a_{m-13} &  \dots & 0 & -1 & a_{m-1m+1} & \dots & a_{m-1n-1} & 1 & a_{m-1n+1} \\
	 1& 1& 1 & \cdots& 1& 0& 1& \cdots& 1& 1& 1& \cdots \\
	 a_{m+11} & a_{m+12} & a_{m+13} &  \dots & a_{m+1m-1} & -1 & 0 & \dots & a_{m+1n-1} & 1 & a_{m+1n+1} \\
	 \vdots& \vdots&\vdots & \vdots& \vdots& \vdots& \vdots& \vdots& \vdots& \vdots& \vdots& \vdots \\
	 a_{n-11} & a_{n-12} & a_{n-13} &  \dots & a_{n-1m-1} & -1 & a_{n-1m+1} & \dots & 0 & 1 & a_{n-1n+2} \\
	 -1& -1 & -1 & \cdots& -1& -1& -1& \cdots& -1& 0& -1& \cdots \\
	 a_{n+11} & a_{n+12} & a_{n+13} &  \dots & a_{n+1m-1} & -1 & a_{n+1m+1} & \dots & a_{n+1n-1} & 1 & 0 \\
	 \vdots& \vdots&\vdots & \vdots& \vdots& \vdots& \vdots& \vdots& \vdots& \vdots& \vdots& \vdots \\
	 \end{array}
	 \right]
	 \]
	 Choose $ {\bf b} = (0,\dots,0,-b_{m},0,\dots,0,b_{n},0,\dots) $
such that $ b_{m},b_{n} > 0 $. One can see that $ b_{k}a_{ki} \leq 0 $  for any pair $(k,i)\in\bn^2$.
\end{exm}

Denote
$$
\begin{array}{ll}
{\bf b}_{\uparrow} = (b_{1},\dots,b_{n},\dots),\ \ \
\mbox{such that } b_{1} \leq \cdots \leq b_{n} \leq \cdots\\
{\bf b}_{\downarrow} = (b_{1},\dots,b_{n},\dots),\ \ \
\mbox{such that } b_{1} \geq \cdots \geq b_{n} \geq \cdots
\end{array}
$$

\begin{thm}\label{thm_Ly_D_up}
Let $V\in\mathcal V^+$ and ${\bf b}_\uparrow\in\ell^\infty$.
Then a functional $ \f_{{\bf b}_\uparrow}$ given by \eqref{uffff} on $S$
is a Lyapunov function for $V$.
\end{thm}
\begin{proof} Let $\frak{f}(V)=(a_{ki})$, then $(a_{ki})\in\mathbb{A}^+$, i.e. $ a_{ki} \geq 0 $ for any $ i >k $. From the proof of Theorem \ref{thm_Ly_da<0}, one finds
\begin{eqnarray}
\f_{{\bf b}_\uparrow}(V(\xb))
= \f_{{\bf b}_\uparrow}(\xb) + \sum\limits_{k=1}^{\infty} b_{k}x_{k}\sum\limits_{i=1}^{\infty}a_{ki}x_{i}  \nonumber.
\end{eqnarray}
Now	taking into account that the series
$\sum\limits_{k=1}^{\infty} b_{k}x_{k}\sum\limits_{i=1}^{\infty}a_{ki}x_{i}$
is absolutely converged
together with $ a_{ki} = - a_{ik} $ and $ b_{k} \leq b_{i} $, whenever $ i > k $, we find
\begin{eqnarray}
\sum\limits_{k=1}^{\infty} b_{k}x_{k}\sum\limits_{i=1}^{\infty}a_{ki}x_{i}
&=& \sum\limits_{k=1}^{\infty}\sum\limits_{i=k+1}^{\infty} x_{k}x_{i} \left( b_{k}a_{ki} + b_{i}a_{ik} \right) \nonumber \\
&=& \sum\limits_{k=1}^{\infty}\sum\limits_{i=k+1}^{\infty} x_{k}x_{i}a_{ki} (b_{k} - b_{i})\nonumber \\[2mm]
&\leq& 0 \nonumber
\end{eqnarray}
Hence,
$$
\f_{{\bf b}_{\uparrow}}(V(\xb)) \leq \f_{{\bf b}_{\uparrow}}(\xb),
$$
which yields that $\lim\limits_{n \rightarrow \infty}\f_{{\bf b}_{\uparrow}}(V^{n}(\xb))$ exists.
This completes the proof.
\end{proof}

Using the same argument, we can prove

\begin{cor}\label{cor_a_ik_-} Let $V\in\mathcal V^-$ and ${\bf b}_\downarrow\in\ell^\infty$.
Then the functional $ \f_{{\bf b}_\downarrow}$ given by \eqref{uffff} on $S$
is a Lyapunov function for $ V $.
\end{cor}

To investigate $\omega_V^{(w)}(\xb_0)$ of Volterra q.s.o. usual Lyapunov functions
may not be applicable. Therefore, we want to introduce quasi Lyapunov function
which is pointwise continuous rather than $\ell^1$-norm continuity.

A pointwise continuous function $\varphi: {\bf B}^+_1\to\mathbb R$ is called a
{\it quasi Lyapunov function} for q.s.o. $V$ if the limit $ \lim\limits_{n
	\rightarrow \infty } \varphi(V^{n}(\xb))$ exists for any initial
point $ \xb\in S$.

Next Theorem describes all linear quasi Lyapunov functions for Volterra operators taken from the classes $\mathcal V^+$, $\mathcal V^-$.

\begin{thm}\label{thm_qLf} Let $V\in\mathcal V^+\cup\mathcal V^-$ and ${\bf b}_\downarrow\in c_0$.
Then a linear functional
$\varphi_{{\bf b}_\downarrow}$ given by \eqref{uffff} on ${\bf B}_1^+$
is a quasi Lyapunov function.
\end{thm}
\begin{proof}
From Lemma \ref{lem5} we immediately conclude that $\varphi_{{\bf b}_\downarrow}$ is pointwise continuous on
${\bf B}^+_1$. One can check that $b_k\geq0$ for any $k\in\bn$, since $\{b_n\}$ is decreasing sequence with
$b_n\to0$.

Let $\frak{f}(V)=(a_{ki})$. Then we have
$$
\varphi_{{\bf b}_\downarrow}(V(\xb))-\varphi_{{\bf b}_\downarrow}(\xb)=
\sum_{k=1}^\infty\sum_{i=k+1}^\infty a_{ki}(b_k-b_i)x_kx_i.
$$
Since the sequence $\{b_n\}$ is decreasing, then one gets
$$
\varphi_{{\bf b}_\downarrow}(V(\xb))-\varphi_{{\bf b}_\downarrow}(\xb)\left\{
\begin{array}{ll}
\geq0, & \mbox{if } V\in\mathcal V^+,\\
\leq0, & \mbox{if } V\in\mathcal V^-.
\end{array}
\right.
$$
This yields that for any $\xb\in S$ the sequence $\{\varphi_{{\bf b}_\downarrow}(V^n(\xb))\}$ is increasing if
$V\in\mathcal V^+$, and it is decreasing if $V\in\mathcal V^-$.

On the other hand, due to $0\leq b_n\leq b_1$ for any $n\geq1$,
one has $0\leq\varphi_{{\bf b}_\downarrow}(\yb)\leq b_1$
for every $\yb\in{\bf B}_1^+$. Consequently,
$$
0\leq\varphi_{{\bf b}_\downarrow}(V^n(\xb))\leq b_1,\ \ \ \forall n\in\mathbb N,\ \forall\xb\in S.
$$
So, we conclude that if $V\in\mathcal V^+\cup\mathcal V^-$ then there exists the limit
$\lim_{n\to\infty}\varphi_{{\bf b}_\downarrow}(V^n(\xb))$ for any
$\xb\in S$.
This implies the required assertion.
\end{proof}
\begin{rem}
From the last Theorem, we conclude that if ${\bf b}\notin c_0$,
then $\varphi_{\bf b}$ is not quasi Lyapunov function,
but it is Lyapunov function.
\end{rem}

\section{Omega Limiting Sets}

In this section, we are interested to study the limiting sets $\omega_V(\xb)$, $\omega_V^{(w)}(\xb)$
and their relation for Volterra operators from the class $\mathcal V^+\cup\mathcal V^-$.
To do so we are going to employ constructed (quasi) Lyapunov functions.
\begin{prop}\label{prop_w=w}
Let $V\in\mathcal V^+\cup\mathcal V^-$ and $\xb_0\in S$.Then the following
statements hold:
\begin{enumerate}
\item[$(i)$] if $V\in\mathcal V^+$ then $\w_V^{(w)}(\xb_0)\subset S$;\\
\item[$(ii)$] if $V\in\mathcal V^-$ then $\w_V^{(w)}(\xb_0)\subset S_r$ for some $r\leq1$.
\end{enumerate}
\end{prop}
\begin{proof} If $\xb_0\in Fix(V)$ then the statements are obvious. Now, let us assume
$\xb_0\in S\setminus Fix(V)$. According to Proposition \ref{Far}
we infer that $\w_{V}^{(w)}(\xb_0)\neq\emptyset$ for any q.s.o. $V$. Moreover,
$\w_{V}^{(w)}(\xb_0)\subset{\bf B}_1^+$.\\
$(i)$ Let $V\in\mathcal V^+$, and pick a point
$\ab\in\w_{V}^{(w)}(\xb_0)$. We want to prove that
$\norm{\ab}=1$.

Assume that $\norm{\ab}<1$.
Due to $\norm{\xb_0}=1$, for a positive number
$\varepsilon=\frac{1-\norm{\ab}}{2}$, there exists an
integer $m\geq1$ such that
$$
\sum_{k=1}^{m}x_k^0>1-\varepsilon.
$$
For a given $m$ let us define a sequence ${\bf b}^{[m]}_{\downarrow}=(\tilde{b}_1,\tilde{b}_2,\dots)$
as follows
\begin{equation}\label{c_k}
\tilde{b}_k=\left\{
\begin{array}{ll}
1, & k\leq m;\\
\frac{1}{2^k}, & k>m.
\end{array}
\right.
\end{equation}
It is clear that ${\bf b}^{[m]}_{\downarrow}\in c_0$ and
Theorem \ref{thm_qLf} implies that $\varphi_{{\bf b}^{[m]}_\downarrow}$ is a quasi
Lyapunov function for $V$. Hence, we have
$$
\varphi_{{\bf b}^{[m]}_\downarrow}(\xb_0)=\sum_{k=1}^{m}x_k^0+\sum_{k=m+1}^\infty\frac{x_k^0}{2^k}\geq\sum_{k=1}^{m}x_k^0>1-\varepsilon
$$
and
$$
\varphi_{{\bf b}^{[m]}_\downarrow}(\ab)=\sum_{k=1}^{m}a_k+\sum_{k=m+1}^\infty\frac{a_k}{2^k}\leq\sum_{k=1}^\infty a_k=\norm{\ab}.
$$
Again from Theorem \ref{thm_qLf} we infer that the sequence $\{\varphi_{{\bf b}^{[m]}_\downarrow}(V^n(\xb_0))\}$
is increasing. Therefore,
\begin{eqnarray*}
\varphi_{{\bf b}^{[m]}_\downarrow}(V^n(\xb_0))-\varphi_{{\bf b}^{[m]}_\downarrow}(\ab)&\geq&\varphi_{{\bf b}^{[m]}_\downarrow}(\xb_0)-
\varphi_{{\bf b}^{[m]}_\downarrow}(\ab)\\
&>&1-\varepsilon-\norm{\ab}\\
&=&\frac{1+\norm{\ab}}{2}>0.
\end{eqnarray*}
This contradicts to the pointwise continuity of $\varphi_{{\bf b}^{[m]}_\downarrow}$ at point $\ab$. So, we conclude
$\norm{\ab}=1$, which yields $\w_V^{(w)}(\xb_0)\subset S$.\\
$(ii)$ Let $V\in\mathcal V^-$. Pick any $\xb,\yb\in\w_V^{(w)}(\xb_0)$. Now we want to show
$\xb=\yb$.

Suppose that $\norm{\xb}>\norm{\yb}$.
For a positive $\varepsilon=\frac{\norm{\xb}-\norm{\yb}}{2}$ there exists an
integer $m\geq1$ such that
$$
\sum_{k=1}^{m}x_k>\norm{\xb}-\varepsilon.
$$
Furthermore, for a given $m\geq1$, let us consider ${\bf b}^{[m]}_\downarrow\in c_0$ given by \eqref{c_k}. Then, Theorem
\ref{thm_qLf} yields that $\varphi_{{\bf b}^{[m]}_\downarrow}$ is a quasi Lyapunov function for $V$, and
there exists $\xi\in[0,1]$ such that
\begin{equation}\label{contrad}
\varphi_{{\bf b}^{[m]}_\downarrow}(V^n(\xb_0))\to\xi,\ \ \ \mbox{as } n\to\infty
\end{equation}

On the other hand, we have
\begin{eqnarray*}
\varphi_{{\bf b}^{[m]}_\downarrow}(\xb)&=&\sum_{k=1}^mx_k+\sum_{k=m+1}^\infty\frac{x_k}{2^k}\\
&\geq&\sum_{k=1}^mx_k>\norm{\xb}-\varepsilon\\
&=&\frac{\norm{\xb}+\norm{\yb}}{2}>\norm{\yb}\\
&\geq&\sum_{k=1}^my_k+\sum_{k=m+1}^\infty\frac{y_k}{2^k}\\
&=&\varphi_{{\bf b}^{[m]}_\downarrow}(\yb).
\end{eqnarray*}

Thus, we have shown that $\varphi_{{\bf b}^{[m]}_\downarrow}(\xb)>\varphi_{{\bf b}^{[m]}_\downarrow}(\yb)$
if $\norm{\xb}>\norm{\yb}$.
This contradicts to \eqref{contrad}, so
$\norm{\xb}=\norm{\yb}$. Hence, $\w_V^{(w)}(\xb_0)\subset S_r$ for some $r\geq0$. Finally,
by Proposition \ref{Far} we conclude that $r\leq1$.
\end{proof}

\begin{thm}\label{omega=omega}
Let $V\in\mathcal V^+\cup\mathcal V^-$ and $\xb_0\in S$. Then the following statements hold:
\begin{enumerate}
\item[$(i)$] if $V\in\mathcal V^+$ then $\w_V(\xb_0)=\w_V^{(w)}(\xb_0)$;
\item[$(ii)$] if $V\in\mathcal V^-$, then $\w_V(\xb_0)=\w_V^{(w)}(\xb_0)$ iff $\w_V(\xb_0)\neq\emptyset$.
\end{enumerate}
\end{thm}

\begin{proof}
$(i)$ Let $V\in\mathcal V^+$. Then according to Proposition
\ref{prop_w=w} we have $\w_V^{(w)}(\xb_0)\subset S$. By Lemma \ref{lem4} one gets $\w_V(\xb_0)=\w_V^{(w)}(\xb_0)$.

$(ii)$ Let $V\in\mathcal V^-$. First we assume that
$\w_V(\xb_0)=\w_V^{(w)}(\xb_0)$. Then Proposition \ref{Far} yields
$\w_V^{(w)}(\xb_0)\neq\emptyset$ which means $\w_V(\xb_0)\neq\emptyset$.

Now suppose that $\w_V(\xb_0)\neq\emptyset$. Due to Lemma \ref{lem1} we have
$\w_V(\xb_0)\subset S$. Therefore, Lemma \ref{lem4} $(i)$ yields
$\w_V(\xb_0)\subset\omega^{(w)}_V(\xb_0)$, so, $\omega^{(w)}_V(\xb_0)\cap S\neq\emptyset$.
Finally, by Proposition \ref{prop_w=w}
we obtain $\w_V^{(w)}(\xb_0)\subset S$. Hence, Lemma \ref{lem4} $(ii)$
implies $\w_V^{(w)}(\xb_0)\subset\w_V(\xb_0)$, which means $\w_V(\xb_0)=\w_V^{(w)}(\xb_0)$.
This completes the proof.
\end{proof}

Now it would be interesting to know the cardinality of
$\omega_V^{(w)}(\xb_0)$. Next result clarifies this question.
\begin{prop}\label{thm_main}
Let $V\in\mathcal V^+\cup\mathcal V^-$. Then $\abs{\w^{(w)}_V(\xb_0)}=1$
 for any $\xb_0\in S$.
\end{prop}
\begin{proof}
It is clear that $\omega_V(\xb_0)=\{\xb_0\}$ for any $\xb_0\in Fix(V)$. So, we prove the
assumption of Theorem only for $\xb_0\in S\setminus Fix(V)$.

Let $\xb_0\in S\setminus Fix(V)$. Take a sequence
$\{{\bf b}^{(n)}\}_{n\geq1}\subset c_0$ defined by
$$
b_k^{(n)}=\left\{
\begin{array}{ll}
\frac{1}{k}, & k\leq n;\\[2mm]
0, & k>n.
\end{array}\right.,\ \ \ k\in\mathbb N
$$
Then Theorem \ref{thm_qLf} yields that
$\varphi_{{\bf b}^{(n)}}$ (for every $n\geq1$) on ${\bf B}_1^+$
is a quasi Lyapunov function for $V\in\mathcal V^+\cup\mathcal V^-$.

Assume that $\xb,\yb\in\w^{(w)}_V(\xb_0)$. Then the argument of the proof of
Proposition \ref{prop_w=w} implies
$$
\varphi_{{\bf b}^{(n)}}(\xb)=\varphi_{{\bf b}^{(n)}}(\yb)\ \ \ \mbox{for any }n\geq1,
$$
which yields $\xb=\yb$. This completes the proof.
\end{proof}

\section{Dynamics of operators from the class $\mathcal V^+$}

In this section we are going to study dynamics of operators taken
from the class $\mathcal V^+$.

We point out that
if $V=Id$, then all points of $S$ are fixed, hence $\w_{V}(\xb_{0})=\{\xb_{0}\}$ for any $ \xb_0\in S $.
Hence, in what follows we always assume that $V\neq Id$.
\begin{thm}\label{thm_LLF_B}
	Let $V\in\mathcal V^+$. Then
	for any initial point $\xb_{0}\in S\setminus Fix(V)$, we have $\w_{V}(\xb_{0})\subset\partial S$.
\end{thm}
\begin{proof} Let $\frak{f}(V)=(a_{ki})$. Due to $(a_{ki})\in\mathbb{A}^+$ there exists a pair $(k_0,i_0)\in\bn^2$ such that
$k_0<i_0$ and $a_{k_0i_0}>0$.

Define a sequence ${{\bf b}}_{\uparrow}=(b_1,b_2,\dots)$ by
\begin{eqnarray}\label{eqn_cho_c_con_PW_>0}
b_k  = \left\{
\begin{array}{ll}
2 -\dfrac{1}{k} & \textmd{ for } k=1,\dots i_{0}, \\ & \\
2 -\dfrac{1}{i_{0}} & \textmd{ for } k \geq i_{0}+1
\end{array}
\right.
\end{eqnarray}
		
Then the functional
$$
\f_{{\bf b}_{\uparrow}}(\xb) = \sum\limits_{k=1}^{\infty}b_{k} x_{k}
$$
is a Lyapunov function for $V$
(see Theorem \ref{thm_Ly_D_up}).
	
Take any $ \xb^{*} \in \w_{V}(\xb_{0}) $, this means there exists a subsequence $\{n_{j}\} $ such that
\begin{equation}\label{yul1}	
V^{n_{j}}(\xb_{0})\stackrel{\norm{\cdot}}{\longrightarrow}\xb^{*},\ \ \ \mbox{as}\ \ j\to\infty.
\end{equation}
The continuity of $\varphi_{{\bf b}_\uparrow}$ implies
$$
\f_{{\bf b}_{\uparrow}} \left(V^{n_{j}}(\xb_{0})\right) \rightarrow \f_{{\bf b}_{\uparrow}}\left(\xb^{*}\right),\ \ \ j\to\infty.
$$
	
Using $ \sum\limits_{i=1}^{\infty}x_{i}=1 $, for any $\xb\in S$, we get
\begin{eqnarray} \label{eqn_LLF_thm_B}
\f_{{\bf b}_{\uparrow}}(\xb)&=&\sum\limits_{k=2}^{\infty}{b}_{k}x_{k} +{b}_{1}
\left( 1- \sum\limits_{k=2}^{\infty}x_{k}\right)\nonumber\\
&=&\sum\limits_{k=2}^{\infty}({b}_{k}-{b}_{1})x_{k}+{b}_{1} \geq {b}_{1}.
\end{eqnarray}
Therefore, $\f_{{\bf b}_{\uparrow}}(\xb^{*})\geq {b}_{1}$.

From \eqref{eqn_cho_c_con_PW_>0}, it follows that
\begin{eqnarray}\label{eqn_2C}
\f_{{\bf b}_{\uparrow} }(\xb) = \sum\limits_{k=1}^{i_{0}} \left(b_{k} - b_{i_0+1} \right) x_{k}
+ b_{i_{0}+1}.
\end{eqnarray}
We consider two cases.
First we assume that $\f_{{{\bf b}}_{\uparrow} }(\xb^{*}) ={b}_{1}$. Since the function \eqref{eqn_2C}
is defined on a compact set
$$
\triangle^{i_{0}} = \left\{ \xb = (x_{1},x_2,\dots,x_{i_0})\in \br^{i_{0}};\ x_{k}\in [0,1],\ 1\leq k\leq i_0\right\}
$$
and $\f_{{{\bf b}}_{\uparrow} }$ is convex on $\triangle^{i_{0}}$,
therefore it attains its unique minimum at $ \xb =\eb_{1} $, which yields $\xb^{*}=\eb_1$, so $\xb^{*}\in\partial S$.
	
Now, we suppose that
\begin{eqnarray}\label{eqn_tre_conve}
\varphi_{{{\bf b}}_{\uparrow}} \left( \xb^{*} \right)>{b}_{1}
\end{eqnarray}
Then from \eqref{yul1} one gets
\begin{eqnarray}
1&=&\lim\limits_{j\rightarrow\infty}\dfrac{\f_{{{\bf b}}_{\uparrow}}
\left(V^{n_j+1}(\xb_{0})\right)-{b}_{1}}{\f_{{{\bf b}}_{\uparrow}}\left( V^{n_j}(\xb_{0})\right)-{b}_{1}} \nonumber \\
&=&\lim\limits_{j\rightarrow \infty} \dfrac{\sum\limits_{k=1}^{\infty}{b}_{k} (V^{n_j+1}(\xb_{0}))_{k}
-{b}_{1}}{\sum\limits_{k=1}^{\infty}{b}_{k} (V^{n_j}(\xb_{0}))_{k}-{b}_{1}} \nonumber \\
&=& \lim\limits_{j\rightarrow \infty} \dfrac{\sum\limits_{k=1}^{\infty}{b}_{k} (V^{n_j}(\xb_{0}))_{k}
\left( 1 + \sum\limits_{i=1}^{\infty} a_{ki} (V^{n_j}(\xb_{0}))_{i}\right)-{b}_{1}}{\sum\limits_{k=1}^{\infty}{b}_{k} (V^{(n_j)}(\xb_{0}))_{k}
-{b}_{1}} \nonumber \\
&=&\lim\limits_{j\rightarrow\infty} \left( 1 +\dfrac{ \sum\limits_{k=1}^{\infty}{b}_{k}
(V^{n_j}(\xb_{0}))_{k} \sum\limits_{i=1}^{\infty} a_{ki}V^{n_j}(\xb_{0})_{i} }{\sum\limits_{k=1}^{\infty}{b}_{k} (V^{n_j}(\xb_{0}))_{k} -
{b}_{1}}\right) \nonumber
\end{eqnarray}
Therefore,
\begin{eqnarray}
0=\lim\limits_{j\rightarrow \infty}\dfrac{\sum\limits_{k=1}^{\infty}{b}_{k}(V^{n_j}(\xb_{0}))_{k}
\sum\limits_{i=1}^{\infty} a_{ki}(V^{n_j}(\xb_{0}))_{i} }{\sum\limits_{k=1}^{\infty}{b}_{k} (V^{n_j}(\xb_{0}))_{k}-
{b}_{1}}\nonumber
\end{eqnarray}
Due to \eqref{eqn_tre_conve} we infer
\begin{eqnarray}\label{eqn_tre_conve-2}
0 = \lim\limits_{j\rightarrow \infty} \sum\limits_{k=1}^{\infty}{b}_{k} (V^{n_j}(\xb_{0}))_{k}
\sum\limits_{i=1}^{\infty} a_{ki}(V^{n_j}(\xb_{0}))_{i}
\end{eqnarray}
On the other hand, if $ \xb^{*} \in ri S $, by taking $m>i_0$,
and using the condition of the theorem, we obtain
	\begin{eqnarray}\label{eqn_LLF_1D}
	\sum\limits_{k=1}^{\infty}{b}_{k} (V^{n_j}(\xb_{0}))_{k}\sum\limits_{i=1}^{\infty} a_{ki}(V^{n_j}(\xb_{0}))_{i}
	&=&  \sum\limits_{k=1}^{\infty} \sum\limits_{i=k+1}^{\infty}(V^{n_j}(\xb_{0}))_{k} (V^{n_j}(\xb_{0}))_{i}
a_{ki} ({b}_{k}-{b}_{i}) \nonumber \\
	&\leq&\sum\limits_{k=1}^{m}\sum\limits_{i=k+1}^{m} (V^{n_j}(\xb_{0}))_{k} (V^{n_j}(\xb_{0}))_{i}
a_{ki}({b}_{k}-{b}_{i})\nonumber \\
	&\leq& (V^{n_{j}}(\xb))_{k_{0}}(V^{n_{j}}(\xb))_{i_{0}} a_{k_{0}i_{0}} ({b}_{k_{0}}-{b}_{i_{0}})
	\end{eqnarray}
Now	taking the limit $j\to\infty$ on \eqref{eqn_LLF_1D} one finds
\begin{eqnarray}\label{eqn_tre_conve-3}
\lim\limits_{{j} \rightarrow \infty} \sum\limits_{k=1}^{\infty}b_{k} (V^{n_j}(\xb_{0}))_{k}
\sum\limits_{i=1}^{\infty} a_{ki}(V^{n_j}(\xb_{0}))_{i} \leq x_{k_{0}}^{*}x_{i_{0}}^{*}
a_{k_{0}i_{0}} ({b}_{k_{0}}-{b}_{i_{0}}) <0,
\end{eqnarray}
which contradicts to \eqref{eqn_tre_conve-2}, hence $\xb^{*}\in \partial S$.
This completes the proof.
\end{proof}
Using the same argument of Theorem \ref{thm_LLF_B}, we can prove the following result.

\begin{cor}\label{cor_LLF_A}
Let $ V\in\mathcal V^+$ and $\frak{f}(V)= \left( a_{ki} \right)$.
If there exist some $ k_{0},i_{0} \in\bn $ such that $ a_{k_{0}i_{0}}>0 $,
then for $ \xb^{*} \in \w_{V}(\xb_{0})  $ we have
\begin{eqnarray}\label{eqn_LLF_A1}
\mbox{either }x^{*}_{k_{0}} = 0 \textmd{ or } x^{*}_{i_{0}} = 0
\end{eqnarray}
\end{cor}
\begin{proof}
Let us assume that
\begin{eqnarray}\label{eqn_LLF_A2}
x_{k_{0}}^{*} >0 \textmd{ and  } x_{i_{0}}^{*} > 0.
\end{eqnarray}
Then from \eqref{eqn_tre_conve-2} and \eqref{eqn_tre_conve-3} one gets the desired statement.
\end{proof}

From Theorem \ref{thm_LLF_B} we infer that the limiting set of trajectory of any operator taken from the class $\mathcal V^+$ belongs to
$\partial S$, but unfortunately we do not know about the location of $\omega_V(\xb)$ in  $\partial S$.  This problem is tricky for the entire class. Next results partially answers to the mentioned  question for some subclass of $\mathcal V^+$. 

\begin{prop} \label{cor_LLF_E}
Let $ V\in\mathcal V^+$ and $\frak{f}(V)= \left( a_{ki} \right)$. 
Assume there exists an integer $n_{0}\geq1$ satisfying $ a_{n_{0}i}>0 $ for all $ i > n_{0} $. Then for any initial point
$ \xb = (0,\dots, 0, x_{n_{0}}, x_{n_{0}+1},\dots) \in S\setminus Fix(V)$ such that $ x_{n_{0}} >0 $,
we have
$$
\w_{V}(\xb)=  \{{\bf e}_{n_{0}} \}
$$
\end{prop}

\begin{proof} First observe that Theorem \ref{omega=omega} implies
$\omega_V(\xb_0)=\omega_V^{(w)}(\xb_0)\in S$ for any
$\xb_0\in S$. Let $ \xb = (0,\dots,0,x_{n_{0}},x_{n_{0}+1}, \dots,) $, we get
\begin{eqnarray}\label{eqn_LLF_1E}
(V(\xb))_{k} = 0 \textmd{ for all } k\in\{1,\dots,n_{0}-1\}.
\end{eqnarray}
Then for $\xb^*\in\omega_V(\xb)$ we have $ x_{i}^{*}=0 $ for any $i\in\{1,\dots,n_{0}-1\}$.
The assumption $ a_{n_{0}i}>0 $, for all $ i > n_{0} $ implies that
\begin{eqnarray*}
(V(\xb))_{n_{0}}&=&x_{n_{0}}
\left( 1 + \sum\limits_{i=1}^{\infty} a_{n_{0}i}x_{i} \right)\\
&=&x_{n_{0}} \left( 1 + \sum\limits_{i=n_{0}+1}^{\infty} a_{n_{0}i}x_{i} \right)\\
&\geq& x_{n_{0}} >0
\end{eqnarray*}
	Therefore, for any $ m \in \bn $
$$
(V^{m+1}(\xb))_{n_{0}} \geq (V^{m}(\xb))_{n_{0}} >0.
$$
Hence, by taking the limit $ m \rightarrow \infty $, we obtain
\begin{eqnarray}\label{eqn_LLF_2E}
x_{n_{0}}^{*} >0
\end{eqnarray}
	
Then Corollary \ref{cor_LLF_A} with \eqref{eqn_LLF_2E} yields
$$
x_{i}^{*}=0 \textmd{ for any } i > n_{0}.
$$
This means $\xb^{*}={\bf e}_{n_{0}}$. The proof is complete.
\end{proof}

\begin{thm}\label{thm_LLF_E}
Let $ V\in\mathcal V^+$ and $\frak{f}(V)= \left( a_{ki} \right)$.
Assume that $ a_{ki}>0 $ for all $ i > k $. Then for any initial point $\xb\in S\setminus Fix(V)$
we have
$$
\w_{V}(\xb)=  \{{\bf e}_{\min\{supp(\xb)\}} \}
$$
\end{thm}
\begin{proof}
Let $\xb\in S\setminus Fix(V)$. Denote
$$
m_0=\min_{i\geq1}\{i: x_i\neq0\}.
$$
Then $a_{m_0i}>0$ for every $i>m_0$, so
Proposition \ref{cor_LLF_E} implies $\omega_V(\xb)=\{{\eb}_{m_0}\}$.
\end{proof}

\section{Dynamics of operators from the class $\mathcal V^-$}

In this section, we are going to study the set $\omega_V(\xb)$ and $\omega_V^{(w)}(\xb)$
for operators
from $\mathcal V^-$.
\begin{thm}\label{thm_LP_con_PW_<0}
Let $V\in\mathcal V^-$ and $\xb_0\in riS\setminus Fix(V)$.
Then $\w_V^{(w)}(\xb_0)\subset\partial S_r$ for some $r\leq1$.
\end{thm}
\begin{proof}
Due to Proposition \ref{prop_w=w} $(ii)$ we have
$\w_V^{(w)}(\xb_0)\subset S_r$ for some $r\leq1$.

Let $\ab\in\w_V^{(w)}(\xb_0)$. Without lost of
generality, we may assume that $\norm{\ab}\neq0$.
Then for any $\varepsilon>0$ we can find $m\geq1$
such that
$$
\sum_{k=1}^ma_k>\norm{\ab}-\varepsilon.
$$
Furthermore, for a given $m\geq1$ we choose ${\bf b}_\downarrow\in\ell^\infty$ as \eqref{c_k}.
Then due to the proof of statement $(ii)$ of Proposition \ref{prop_w=w},
$\varphi_{{\bf b}_\downarrow}$ is a quasi Lyapunov function for $V$, and
the sequence $\{\varphi_{{\bf b}_\downarrow}(V^n(\xb))\}$ is decreasing for any $\xb\in S$.   	

It is easy to check that $\varphi_{{\bf b}_\downarrow}(\yb)\leq\norm{\ab}$
for any $\yb\in S_{\norm{\ab}}$. Moreover,
$\varphi_{{\bf b}_\downarrow}(\yb)=\norm{\ab}$ if and only if $y_k=0$ for any $k>m$. So, if
$\varphi_{{\bf b}_\downarrow}(\ab)=\norm{\ab}$ we have $\ab\in\partial S_{\norm{\ab}}$.

Let $\frak{f}(V)= (a_{ki})$, then $ (a_{ki})\in\mathbb{A}^-$. Now we suppose that $\varphi_{{\bf b}_\downarrow}(\ab)<\norm{\ab}$.
Assume that $\ab\in riS_{\norm{\ab}}$.
Since $\ab\in\w_V^{(w)}(\xb_0)$, there is subsequence
$\{n_j\}$ such that $V^{n_j}(\xb_0)\stackrel{\mathrm{p.w.}}{\longrightarrow}\ab$ as $j\to\infty$.
This implies that
\begin{eqnarray}\label{eqn_lim_=1}
\lim\limits_{j\rightarrow \infty} \dfrac{ \f_{{\bf b}_{\downarrow}}(V^{n_{j}+1}(\xb_0))
-\norm{\ab}}{\f_{{{\bf b}}_{\downarrow} }(V^{n_j}(\xb_0))-\norm{\ab}} = 1
\end{eqnarray}
From
\begin{eqnarray}
\lim\limits_{j\rightarrow \infty} \dfrac{ \f_{{{\bf b}}_{\downarrow} }(V^{n_j+1}(\xb_0)) - \norm{\ab}}
{\f_{{{\bf b}}_{\downarrow} }(V^{n_j}(\xb_0))  - \norm{\ab}}
&=& \lim\limits_{j\rightarrow \infty} \dfrac{\sum\limits_{k=1}^{\infty}{b}_{k}
(V^{n_j}(\xb_0))_{k} \left( 1+ \sum\limits_{i=1}^{\infty} a_{ki} (V^{n_j}(\xb_0))_{i}\right)-
\norm{\ab}}{\sum\limits_{k=1}^{\infty}{b}_{k} (V^{n_j}(\xb_0))_{k}
- \norm{\ab} } \nonumber \\
&=& \lim\limits_{j \rightarrow \infty} \left(   1 + \dfrac{ \sum\limits_{k=1}^{\infty} {b}_{k} (V^{n_j}(\xb_0))_{k}
\sum\limits_{i=1}^{\infty} a_{ki} (V^{n_j}(\xb_0))_{i}}{\sum\limits_{k=1}^{\infty}{b}_{k}
(V^{n_j}(\xb_0))_{k} - \norm{\ab}} \right) \nonumber
\end{eqnarray}
and \eqref{eqn_lim_=1} one finds
$$
\lim\limits_{j \rightarrow \infty} \dfrac{ \sum\limits_{k=1}^{\infty} {b}_{k} (V^{n_j}(\xb_0))_{k}
\sum\limits_{i=1}^{\infty} a_{ki} (V^{n_j}(\xb_0))_{i}}
{\sum\limits_{k=1}^{\infty} {b}_{k} (V^{n_j}(\xb_0))_{k} - \norm{\ab} } =0
$$
Due to $ \f_{{{\bf b}}_{\uparrow} }(\ab) <\norm{\ab}$, we conclude
\begin{eqnarray}\label{eqn_lim=0}
\lim\limits_{n \rightarrow \infty} \sum\limits_{k=1}^{\infty} {b}_{k} (V^{n_j}(\xb_0))_{k}
\sum\limits_{i=1}^{\infty} a_{ki} (V^{n_j}(\xb_0))_{i} = 0.
\end{eqnarray}
Now, denote
$$
\d= \dfrac{a_{k_{0}}a_{i_{0}}}{2} a_{k_{0}i_{0}} \left( {b}_{k_{0}} - {b}_{i_{0}} \right).
$$
It is clear that $ \d < 0 $, since ${b}_{k_{0}} < {b}_{i_{0}}$.
From \eqref{eqn_lim=0} one finds $ N_{1} \in \bn $ such that
\begin{eqnarray} \label{eqn_more_d4}
\sum\limits_{k=1}^{\infty} {b}_{k} (V^{n_j}(\xb_0))_{k} \sum\limits_{i=1}^{\infty} a_{ki} (V^{n_j}(\xb_0))_{i} > \dfrac{\d}{4}
\end{eqnarray}
for all $ n_j \geq N_{1} $.
Now we choose $ N_{2} \in \bn $ such that $ N_{2} \geq N_{1} $ and
$$
\abs{ (V^{N_{2}}(\xb_0))_{k_{0}}  (V^{N_{2}}(\xb_0))_{i_{0}}  - a_{k_{0}}a_{i_{0}} } < \dfrac{a_{k_{0}}a_{i_{0}}}{2}.
$$
This yields that
\begin{eqnarray}\label{eqn_6c_con_PW_>0}
(V^{N_{2}}(\xb_0))_{i_{0}}  (V^{N_{2}}(\xb_0))_{k_{0}}   >\dfrac{a_{k_{0}}a_{i_{0}}}{2}
\end{eqnarray}

Since $ V^{N_{2}}(\xb_0) \in S $, then one can choose $ \widetilde{N}  $ such that
$$
\sum\limits_{k = \widetilde{N}+1}^{\infty}(V^{N_{2}}(\xb_0))_{k}\leq \dfrac{-\d}{6}
$$
and $ \overline{N} $ such that $ \overline{N}>i_{0}>k_{0} $. Put
$$
N = \max\{\widetilde{N},\overline{N}\}
$$
By simple calculation, one has
\begin{eqnarray}
\sum\limits_{k=1}^{\infty} {b}_{k} (V^{N_{2}}(\xb_0))_{k}
\sum\limits_{i=1}^{\infty} a_{ki} (V^{N_{2}}(\xb_0))_{i}
&=& \sum\limits_{k=1}^{\infty}{b}_{k} (V^{N_{2}}(\xb_0))_{k}
\bigg( \sum\limits_{i=1}^{N}a_{ki}(V^{N_{2}}(\xb_0))_{i} \nonumber\\[2mm]
&&+
\sum\limits_{i=N+1}^{\infty}a_{ki}(V^{N_{2}}(\xb_0))_{i} \bigg) \nonumber \\
&=& \sum\limits_{k=1}^{\infty}{b}_{k} (V^{N_{2}}(\xb_0))_{k} \sum\limits_{i=1}^{N}a_{ki}(V^{N_{2}}(\xb_0))_{i}\nonumber \\
&&+\sum\limits_{k=1}^{\infty}{b}_{k} (V^{N_{2}}(\xb_0))_{k} \sum\limits_{i=N+1}^{\infty}a_{ki}(V^{N_{2}}(\xb_0))_{i}
\end{eqnarray}	
Denote
\begin{eqnarray*}
&&I_{1} = \sum\limits_{k=1}^{\infty}{b}_{k} (V^{N_{2}}(\xb_0))_{k}
\sum\limits_{i=1}^{N}a_{ki}(V^{N_{2}}(\xb_0))_{i}, \\[2mm]
&&I_{2} = \sum\limits_{k=1}^{\infty}{b}_{k} (V^{N_{2}}(\xb_0))_{k}
\sum\limits_{i=N+1}^{\infty}a_{ki}(V^{N_{2}}(\xb_0))_{i}
\end{eqnarray*}

Let us estimate $ I_{1} $ and $ I_{2} $ one by one.
Taking into account $ I_{1} $, then one has
\begin{eqnarray}
I_{1} &=&\sum\limits_{k=1}^{\infty}\sum\limits_{i=1}^N ({b}_{k}-{b}_{i})a_{ki}(V^{N_{2}}(\xb_0))_{k}(V^{N_{2}}(\xb_0))_{i}\nonumber\\
&&+\sum\limits_{k=N+1}^{\infty}{b}_{k} (V^{N_{2}}(\xb_0))_{k} \sum\limits_{i=1}^{N}a_{ki}(V^{N_{2}}(\xb_0))_{i} \nonumber
\end{eqnarray}
Denote
\begin{eqnarray*}
&&II_{1} = \sum\limits_{k=1}^{\infty}\sum\limits_{i=1}^N (b_{k}-{b}_{i})a_{ki}(V^{N_{2}}(\xb_0))_{k}(V^{N_{2}}(\xb_0))_{i}, \\
&&II_{2} = \sum\limits_{k=N+1}^{\infty}{b}_{k} (V^{N_{2}}(\xb_0))_{k}
\sum\limits_{i=1}^{N}a_{ki}(V^{N_{2}}(\xb_0))_{i}
\end{eqnarray*}

Due to $ a_{ki}\geq 0 $ for all $ i >k $ and \eqref{eqn_6c_con_PW_>0} one gets
\begin{eqnarray}\label{eqn_7c_con_PW_>0}
II_{1} &=&\sum\limits_{k=1}^{N-1}\sum\limits_{i=k+1}^N ({b}_{k}-{b}_{i})a_{ki}
(V^{N_{2}}(\xb_0))_{k}(V^{N_{2}}(\xb_0))_{i} \nonumber \\
&&+\sum\limits_{k=N+1}^{\infty}\sum\limits_{i=1}^N ({b}_{k}-{b}_{i})a_{ki}
(V^{N_{2}}(\xb_0))_{k}(V^{N_{2}}(\xb_0))_{i} \nonumber \\[2mm]
&\leq& (V^{N_{2}}(\xb_0))_{k_{0}}(V^{N_{2}}(\xb_0))_{i_{0}}a_{k_{0}i_{0}}({b}_{k_{0}}-{b}_{i_{0}}) \nonumber \\[2mm]
&\leq& \dfrac{a_{k_{0}} a_{i_{0}}}{2} a_{k_{0}i_{0}}(b_{k_{0}} -{b}_{i_{0}}) = \d <0
\end{eqnarray}
Taking into account $ II_{2} $ we have
\begin{eqnarray}\label{eqn_9c_con_PW_>0}
\abs{II_{2}} &=& \abs{ \sum\limits_{k=N+1}^{\infty}{b}_{k} (V^{N_{2}}(\xb_0))_{k}
\sum\limits_{i=1}^{N}a_{ki}(V^{N_{2}}(\xb_0))_{i} } \nonumber \\
&\leq&  2 \sum\limits_{k=N+1}^{\infty} (V^{N_{2}}(\xb_0))_{k} \sum\limits_{i=1}^{N}
(V^{N_{2}}(\xb_0))_{i} \nonumber \\
&\leq& 2 \sum\limits_{k=N+1}^{\infty} (V^{N_{2}}(\xb_0))_{k} \nonumber \\
&\leq& \dfrac{-\d}{3}
\end{eqnarray}
Hence from \eqref{eqn_7c_con_PW_>0} and \eqref{eqn_9c_con_PW_>0}, we obtain
\begin{eqnarray}\label{eqn_10c_con_PW_>0}
I_{1} = II_{1}+II_{2} \leq \d - \dfrac{\d}{3} = \dfrac{2\d}{3} <0
\end{eqnarray}

From $ I_{2} $ we find
\begin{eqnarray}\label{eqn_11c_con_PW_>0}
\abs{I_{2}} &=& \abs{ \sum\limits_{k=1}^{\infty}{b}_{k}
(V^{N_{2}}(\xb_0))_{k} \sum\limits_{i=N+1}^{\infty}a_{ki}
(V^{N_{2}}(\xb_0))_{i} } \nonumber \\
&\leq& 2 \sum\limits_{k=1}^{\infty}
(V^{N_{2}}(\xb_0))_{k} \sum\limits_{i=N+1}^{\infty}(V^{N_{2}}(\xb_0))_{i} \nonumber \\
&\leq& 2 \sum\limits_{i=N+1}^{\infty} (V^{N_{2}}(\xb_0))_{i}
\sum\limits_{k=1}^{\infty}(V^{N_{2}}(\xb_0))_{k} \nonumber \\
&\leq& \dfrac{-\d}{3}
\end{eqnarray}

Therefore, from \eqref{eqn_10c_con_PW_>0} and \eqref{eqn_11c_con_PW_>0}, one gets
\begin{eqnarray}
\sum\limits_{k=1}^{\infty} {b}_{k} (V^{N_{2}}(\xb_0))_{k}
\sum\limits_{i=1}^{\infty} a_{ki} (V^{N_{2}}(\xb_0))_{i} \leq
\dfrac{2\d}{3} - \dfrac{\d}{3} = \dfrac{\d}{3} < 0
\end{eqnarray}

From \eqref{eqn_more_d4} it follows that
$$
\dfrac{\d}{4}< \sum\limits_{k=1}^{\infty} {b}_{k} (V^{N_{2}}(\xb_0))_{k}
\sum\limits_{i=1}^{\infty} a_{ki} (V^{N_{2}}(\xb_0))_{i} < \dfrac{\d}{3}
$$
which yields $ \d >0 $. This contradicts to the negativity of $\delta$,
hence $ \ab\in \partial S_{\norm{\ab}}$.
\end{proof}

Using the same argument of Theorem \ref{thm_LP_con_PW_<0}, we can prove the following one.
\begin{cor}\label{cor_LLF_ome_<0}
	Let $ V\in\mathcal V^{-}$ and $\frak{f}(V)= \left( a_{ki} \right) $.
	If there exists a pair $(k_{0},i_{0})\in\bn^2$ such that $ a_{k_{0}i_{0}}<0 $,
then for $ \xb^*\in \w_{V}^{(w)}(\xb_{0})  $ we have
	\begin{eqnarray}\label{eqn_LLF_A1}
\mbox{either }	x^*_{k_{0}} = 0 \textmd{ or } x^*_{i_{0}} = 0
	\end{eqnarray}
\end{cor}

According to Proposition \ref{prop_w=w} and Theorem \ref{thm_LP_con_PW_<0} we have the following result:
\begin{thm}\label{thm_say}
Let $V\in\mathcal V^-$ and $\xb_0\in S\setminus Fix(V)$. Then the following statements hold:
\begin{enumerate}
\item[$(i)$] If $\w_V(\xb_0)\neq\emptyset$
then $\w_V(\xb_0)=\w_V^{(w)}(\xb_0)\subset\partial S$;\\
\item[$(ii)$] If $\w_V(\xb_0)=\emptyset$ then $\w_V^{(w)}(\xb_0)\subset\partial S_r$ for some $r<1$.
\end{enumerate}
\end{thm}

We stress that Theorem \ref{thm_say} states that a limit point of the set
$\{V^n(\xb)\}$ w.r.t. pointwise convergence belongs to
$\partial S_r$ for some $r\leq1$. On the other hand,
it does not give a relation between $\xb\in S$ and $r\leq1$. So, it would be better
if we are able to find that relation for a given $V\in\mathcal V^-$.
Next result sheds some light to this question.
\begin{thm}\label{thm_LLF_E_<0}
Let $ V\in\mathcal V^-$  and $\frak{f}(V)= \left( a_{ki} \right) $.
Assume that $a_{ki}<0$ for all $k<i$
then for every $\xb\in S$ the following statements hold:
\begin{enumerate}
\item[$(i)$] if $|supp(\xb)|<\infty$ then $\w_{V}^{(w)}(\xb)=\{{\bf e}_{\max\{supp(\xb)\}}\}$;
\\
\item[$(ii)$] if $\sup\limits_{i>k}\{a_{ki}\}<0$ for all $ k\geq1$ and $|supp(\xb)|=\infty$, then
$$
\w_{V}^{(w)}(\xb)=\{{\bf 0}\}.
$$
\end{enumerate}
\end{thm}
\begin{proof}
$(i)$ Let $|supp(\xb)|<\infty$. Suppose that $m_0=\max\{supp(\xb)\}$. Without lost of
generality, we may assume that $x_{m_0}\neq1$. It yields that $m_0>1$ and
the existence of $k<m_0$ such that $x_k\neq0$. Denote $\alpha=\max_{i<m_0}\{a_{im_0}\}$.
It is clear that $-1\leq\alpha<0$. Then, one gets
\begin{eqnarray}\label{tengsiz1}
\sum_{i<m_0}(V^{n+1}(\xb))_i&=&\sum_{i<m_0}(V^{n}(\xb))_i+\sum_{j=m_0}^\infty\sum_{i<m_0}a_{ij}(V^{n}(\xb))_i(V^{n}(\xb))_j\nonumber\\[2mm]
&=&\sum_{i<m_0}(V^{n}(\xb))_i+\sum_{i<m_0}a_{im_0}(V^{n}(\xb))_i(V^{n}(\xb))_{m_0}\nonumber\\[2mm]
&\leq&\sum_{i<m_0}(V^{n}(\xb))_i\left(1+\alpha-\alpha\sum_{i<m_0}(V^{n}(\xb))_i\right)
\end{eqnarray}
Let us consider a function $f_\alpha:[0,1)\to[0,1)$ given by
 $$
 f_\alpha(\xi)=\xi(1+\alpha-\alpha\xi),
 $$
 where, $\alpha\in[-1,0)$. It is easy to check that $f$ has a unique fixed point $\xi_0=0$. Moreover, for any
 initial point $\xi\in(0,1)$ trajectory $\{f^n(\xi)\}$ tends to $0$. So, keeping in mind this fact
 from \eqref{tengsiz1} we conclude that
 $$
 \sum_{i<m_0}(V^{n}(\xb))_i\to0,\ \ \mbox{as}\ n\to\infty.
 $$
 Consequently, from $\sum\limits_{i=1}^{m_0}(V^n(\xb))_{i}=1$ one finds
 $$
 (V^n(\xb))_{m_0}\to1,\ \ \mbox{as}\ n\to\infty.
 $$
 This means that $V^n(\xb)\stackrel{\norm{\cdot}}{\longrightarrow}{\bf e}_{m_0}$.

$(ii)$
Assume that $\sup\limits_{i>k}\{a_{ki}\}<0$ for all $ k\geq1$. Let us define a sequence $\{\alpha_m\}_{m\geq1}$
as follows:
$$
\alpha_m=\max_{k\leq m}\{\sup_{i>k}\{a_{ki}\}\}.
$$
It is clear that $\alpha_m<0$ for every $m\geq1$. Then for any $m\geq1$ we have
\begin{eqnarray}\label{tengsiz}
\sum_{i=1}^m(V^{n+1}(\xb))_i&=&\sum_{i=1}^m(V^{n}(\xb))_i+\sum_{j>m}\sum_{i=1}^ma_{ij}(V^{n}(\xb))_i(V^{n}(\xb))_j\nonumber\\[2mm]
&\leq&\sum_{i=1}^m(V^{n}(\xb))_i+\alpha_m\sum_{j>m}\sum_{i=1}^m(V^{n}(\xb))_i(V^{n}(\xb))_j\nonumber\\[2mm]
&=&\sum_{i=1}^m(V^{n}(\xb))_i\left(1+\alpha_m\sum_{j>m}(V^{n}(\xb))_j\right)\nonumber\\[2mm]
&=&\sum_{i=1}^m(V^{n}(\xb))_i\left(1+\alpha_m-\alpha_m\sum_{i=1}^m(V^{n}(\xb))_i\right)
\end{eqnarray}
For any $m\geq1$ on a unit interval we define a sequence $y^{(n)}_m$  as follows
$$
y^{(n)}_m=\sum_{i=1}^m(V^{n}(\xb))_i.
$$
Then from \eqref{tengsiz} one has
$$
y^{(n+1)}_m\leq y^{(n)}_m\left(1+\alpha_m-\alpha_m y^{(n)}_m\right)
$$
Since $\alpha_m\in[-1,0)$ again using the dynamics of function $f_{\alpha_m}$,
we conclude that $y^{(n)}_m\to0$ as $n\to\infty$. Hence,
$$
(V^{n}(\xb))_i\to0,\ \ \ i\leq m.
$$
The arbitraryness of $m\geq1$ yields $V^n(\xb)\stackrel{\mathrm{p.w.}}{\longrightarrow}\bf0$ as $n\to\infty$.
The proof is complete.
\end{proof}

The proved Theorem suggest the following conjecture:
\begin{conj}
Let $ V\in\tilde{\mathcal V}^-$  and and $\frak{f}(V)= \left( a_{ki} \right) $.
If $a_{ki}<0$ for all $k<i$
then for every $\xb\in S$ the following statements hold:
\begin{enumerate}
\item[$(i)$] if $|supp(\xb)|<\infty$ then $\w_{V}^{(w)}(\xb)=\{{\bf e}_{\max\{supp(\xb)\}}\}$;
\item[$(ii)$] if $|supp(\xb)|=\infty$ then $\w_{V}^{(w)}(\xb)=\{{\bf 0}\}$.
\end{enumerate}
\end{conj}

\section{Proofs of main results}

In this section we are going to prove the main results formulated
in Section 1. Before start proofs we need the following
auxiliary result.
\begin{prop}\label{thm_oxir}
Let $V\in\tilde{\mathcal V}^+\cup\tilde{\mathcal V}^-$.
Then for any $\xb_0\in S\setminus Fix(V)$
the following statements hold:
\begin{enumerate}
\item[$(i)$]
if $V\in\tilde{\mathcal V}^+$ then $\omega_{V}^{(w)}({\xb}_0)\in\partial S$;\\
\item[$(ii)$] if $V\in\tilde{\mathcal V}^-$ then there exists $r=r(\xb_0)\leq1$ such that
$\omega_{V}^{(w)}({\xb}_0)\in\partial S_r$.
\end{enumerate}
\end{prop}
\begin{proof}
$(i)$ Let $V\in\tilde{\mathcal V}^+$. There exists $k_0\geq1$ such that $a_{ki}=0$
for any $k<k_0\leq i$ and $a_{ki}\geq0$ for every $k_0\leq k<i$.
Without lost of generality we may assume that
$k_0\geq2$. Indeed, if $k_0=1$ then we have $V\in\mathcal V^+$,
then the statement follows from Proposition \ref{thm_main}.

So, $k_0\geq2$ and take an arbitrary
$\xb_0\in S\setminus{Fix(V)}$. Let us denote
$$
\begin{array}{ll}
\yb_0=(y_1^{(0)},\dots,y_{k_0-1}^{(0)}):=(x_1^{(0)},\dots,x_{k_0-1}^{(0)}),\\[2mm]
\zb_0=(z_1^{(0)},z_2^{(0)},\dots):=(x_{k_0}^{(0)},x_{k_0+1}^{(0)},\dots),\\[2mm]
r_1=\sum\limits_{i=1}^{k_0-1}y_{i}^{(0)},\ r_2=\sum\limits_{i=1}^\infty z_{i}^{(0)}.
\end{array}
$$
It is clear that $r_1+r_2=1$ and $\yb_0\in S_{r_1}^{k_0-1}$, $\zb_0\in S_{r_2}$. We note that

Let us consider several cases w.r.t. $r_1,r_2$.

{\bf Case $r_1=0$}. Then $r_2=1$, so $\zb_{0}\in S$. Now, let us define
skew-symmetric matrix $(\tilde{a}_{ki})$ given by
\begin{equation}\label{tilde}
\tilde{a}_{ki}=a_{k+k_0-1,i+k_0-1},\ \ \ \forall k,i\geq1.
\end{equation}
Then corresponding Volterra operator $\tilde{V}$ belongs to $\mathcal V^+$.
Due to Theorem \ref{thm_LLF_B} we infer $\omega_{\tilde{V}}^{(w)}(\zb_0)\in\partial S$.
Then, from the following
$$
(V^n\xb_0)_k=\left\{
\begin{array}{ll}
0, & k<k_0\\
(\tilde{V}\zb_0)_{k-k_0+1}, & k\geq k_0
\end{array}
\right.
$$
we conclude that $\omega_V^{(w)}(\xb_0)\in\partial S$.

{\bf Case $r_2=0$}. In this case we have $r_1=1$. On the finite dimensional simplex $S^{k_0-1}$,
we consider Volterra q.s.o. $\hat{V}$ with skew-symmetric matrix $(\hat{a}_{ki})_{k,i\geq1}^n$ given by
\begin{equation}\label{hat}
\hat{a}_{ki}=a_{ki},\ \ \ 1\leq k,i\leq k_0-1.
\end{equation}
Due to Theorem \ref{rg1} we have $\omega_{\hat{V}}(\yb_0)\in\partial S^{k_0-1}$. Keeping in mind this fact,
from
$$
(V^n\xb_0)_k=\left\{
\begin{array}{ll}
(\hat{V}\yb_0)_k, & k<k_0\\
0, & k\geq k_0
\end{array}
\right.
$$
we infer $\omega_V^{(w)}(\xb_0)\in\partial S$.

{\bf Case $0<r_1,r_2<1$}. Then we have $\frac{\yb_0}{r_1}\in S^{k_0-1}$ and $\frac{\zb_0}{r_2}\in S$.
Consequently, from Theorem \ref{rg1} one has $\omega_{\hat{V}}(\frac{\yb_0}{r_1})\in\partial S^{k_0-1}$ and
Theorem \ref{thm_LLF_B} yields $\omega_{\tilde{V}}^{(w)}(\frac{\zb_0}{r_2})\in\partial S$.
Finally, from
$$
(V^n\xb_0)_k=\left\{
\begin{array}{ll}
(\hat{V}\yb_0)_k, & k<k_0\\
(\tilde{V}\zb_0)_{k-k_0+1}, & k\geq k_0
\end{array}
\right.
$$
we conclude that $\omega_V^{(w)}(\xb_0)\in\partial S$.

$(ii)$ The statement $(ii)$ can be proved by the same
argument as $(i)$.
\end{proof}

Now we are ready to proceed the proofs of
main results.

\begin{proof}[Proof of Theorem \ref{thm_asos}] The proof immediately
follows from Theorems \ref{thm_asos3} and \ref{thm_oxir}. \end{proof}

{\begin{proof}[Proof of Theorem \ref{thm_asos2}]  Let $V\in\tilde{\mathcal V}^+\cup\tilde{\mathcal V}^-$.
This means that there exists $k_0\geq1$ such that $a_{ki}=0$
for any $k<k_0\leq i$ and $a_{ki}\geq0$ for every $k_0\leq k<i$.

If $k_0=1$ then $V\in\mathcal V^+\cup\mathcal V^-$. So, the statement of the theorem immediately follows from
Proposition \ref{thm_main}.

Let us assume that $k_0>1$.
Then, due to the argument of Theorem \ref{thm_oxir} there exist two
operators  $\hat{V}$ and $\tilde{V}$
defined on
$S^{k_0-1}$ and $S$, respectively, such that
$V$ can be represented by
$$
(V(\xb_0))_k=\left\{
\begin{array}{ll}
(\hat{V}(\yb_0))_k, & \mbox{if }k<k_0\\
(\tilde{V}(\zb_0))_k, & \mbox{if }k\geq k_0
\end{array}
\right.
$$
where $\xb_0\in S$.

According to Theorem \ref{rg1} one has
\begin{equation}\label{123}
|\omega_{\hat{V}}(\yb_0)|=1\vee\infty.
\end{equation}
Furthermore, since $\tilde{V}\in\mathcal V^+\cup\mathcal V^-$ it then  follows from
Proposition \ref{thm_main} that
\begin{equation}\label{1234}
|\omega_{\tilde{V}}^{(w)}(\zb_0)|=1.
\end{equation}
Hence, \eqref{123},\eqref{1234} together with Theorem \ref{omega=omega}
imply the assertions $(i)$ and $(ii)$.

\end{proof}

\begin{proof}[Proof of Theorem \ref{thm_asos3}]

Thanks to Theorems \ref{omega=omega}, \ref{thm_LLF_B} and Proposition \ref{thm_main} we obtain the assertion $(i)$.
Using Theorems \ref{omega=omega}, \ref{thm_LP_con_PW_<0} and Proposition \ref{thm_main} one can prove $(ii)$.
\end{proof}

\begin{proof}[Proof of Theorem \ref{asos4}]  $(i)$ From Proposition \ref{thm_main} we infer that $V$ is weak ergodic at any point of $S$.

$(ii)$ Let $V\in\mathcal V^+$. Then, Theorem \ref{omega=omega} $(i)$ and
Proposition \ref{thm_main} yield $|\omega_V(\xb_0)|=1$ for any $\xb_0\in S$. This implies
the ergodicity of $V$ at $\xb_0$.

$(iii)$ Let $V\in\mathcal V^-$. At first, we suppose that $\omega_{V}(\xb_0)\neq\emptyset$. Then, from Theorem
\ref{omega=omega} and Proposition \ref{thm_main} we find $|\omega_V(\xb_0)|=1$. Hence, $V$ is ergodic at $\xb_0$.

Now, let us assume $\omega_V(\xb_0)=\emptyset$. We are going to
establish that $V$ is not ergodic at $\xb_0$. Suppose that $V$ is ergodic at $\xb_0$. This means that
there exists $\ab\in S$ such that
$$
\frac{1}{n}\sum_{k=0}^nV^k(\xb_0)\stackrel{\norm{\cdot}}{\longrightarrow}\ab,\ \ \mbox{as }n\to\infty.
$$
Then, Lemma \ref{lem4} implies
\begin{equation}\label{ddd}
\frac{1}{n}\sum_{k=0}^nV^k(\xb_0)\stackrel{\mathrm{p.w.}}{\longrightarrow}\ab,\ \ \mbox{as }n\to\infty.
\end{equation}
Hence, from Proposition \ref{thm_main} together with \eqref{ddd} we obtain
$$
V^n(\xb_0)\stackrel{\mathrm{p.w.}}{\longrightarrow}\ab,\ \ \mbox{as }n\to\infty.
$$
From the last one and noting $\omega_V(\xb_0)=\emptyset$ one gets $\norm{\ab}<1$. This contradicts
to $\ab\in S$. So, we conclude that $V$ is not ergodic at $\xb_0$.
\end{proof}

\section*{Acknowledgments}
The present work is supported by the UAEU "Start-Up" Grant, No.
31S259.


\begin{thebibliography}{92}
	
    \bibitem{AL} Akin E., Losert V. Evolutionary dynamics of zero-sum games. {\it J. Math. Biol.} {\bf 20}(1984), 231--258.
	
	\bibitem{1} Bernstein S.N., The solution of a mathematical problem concerning the theory of heredity. \textit{Ucheniye-Zapiski N.-I. Kaf. Ukr. Otd. Mat.} 1 (1924), 83--115 (Russian).

    \bibitem{BFL} Boyarsky  A., Gora P., Lioubimov V., Snap-back repellers and scrambled sets in general topological
spaces, \textit{Nonlinear Analysis}  {\bf 43} (2001),  591--604.

    \bibitem{Ch}  Cheskidov A., Global attractors of evolutionary systems, \textit{J. Dyn. Diff. Equat.}  {\bf 21} (2009),  249--268.

    \bibitem{dahl} Dahlberg C., \textit{Mathematical methods for population genetics}, Interscience Publishing, New York, 1948.
    	
    \bibitem{fish1} Fisher M.E., Goh B.S., Stability in a class of discrete-time models of
    interacting populations. {\it J. Math. Biol.} {\bf 4}(1977), 265--274.

    \bibitem{ngrguj} Ganikhodjaev N.N., Ganikhodjaev R.N., Jamilov U.U., Quadratic stochastic operators and zero-sum game
dynamics, {\it Ergod. Th. Dynam. Sys.} {\bf35}(5), (2015) 1443--1473.
		
    \bibitem{R.gani_tournment} Ganikhodzhaev, R. N. Quadratic stochastic operators, Lyapunov functions, and tournaments, \textit{Russian Acad.  Sci. Sbornik Math.}, {\bf 76}(1993), 489--506.

	\bibitem{6} Ganikhodzhaev R., Mukhamedov F., Rozikov U., Quadratic stochastic operators and processes: results and open problems, \textit{Infin. Dimens. Anal. Quantum Probab. Relat.Top.} {\bf
		14}(2011) 270--335.

    \bibitem{HHJ} Hofbauer J., Hutson V., Jansen W., Coexistence for systems governed by difference equations of
Lotka-Volterra type, \textit{J. Math. Biol.} {\bf 25} (1987), 553--570.

	\bibitem{hsbook} Hofbauer J., Sigmund K., \textit{Evolutionary Games and Population Dynamics}, Cambridge Univ.  Press, Cambridge, 1998.
		
    \bibitem{kest} Kesten H., Quadratic transformations: a model for population growth, {\it Adv. Appl. Probab.} {\bf2}(1) 1970, 1--82.

	\bibitem{11} Lyubich Yu.I., \textit{Mathematical structures in population genetics},  Springer-Verlag, 1992.

   \bibitem{M2000} Mukhamedov  F. M.  On infinite dimensional Volterra operators,
\textit{Russian Math. Surveys} {\bf 55}(2000), 1161--1162

    \bibitem{Far_Has_Temir} Mukhamedov F., Akin H.,  Temir S. On infinite dimensional quadratic Volterra operators, \textit{J.  Math. Anal. Appl.}  {\bf 310} (2005), 533--556.
	
	\bibitem{MG2015} Mukhamedov F., Ganikhodjaev N. \textit{Quantum Quadratic Operators and Processes},
	Lect. Notes Math. {\bf 2133}(2015), Springer,  2015.
	
	\bibitem{Nag} Nagylaki T., Evolution of a large population under gene conversion, \textit{Proc. Natl. Acad. Sci. USA} {\bf  80}
(1983), 5941--5945.
	
    \bibitem{20} Narendra S.G., Samaresh C.M., Elliott W.M., On the Volterra and other nonlinear moldes of interacting populations, \textit{Rev. Mod. Phys.} {\bf 43} (1971), 231--276.

    \bibitem{plank} Plank M., Losert V., Hamiltonian structures for the $n$-dimensional Lotka-Volterra equations.
    {\it J. Math. Phys.} {\bf 36}, (1995) 3520--3543.

    \bibitem{SC} Shi, Y., Chen, G., Chaos for discrete dynamical systems in complete metric spaces, \textit{Chaos,
Solitons \& Fractals} {\bf 22} (2004),  555--571.
	
	\bibitem{25} Takeuchi Y., \textit{Global dynamical properties of Lotka-Volterra systems}, World Scientific, Singapore, 1996.	

	\bibitem{udwad} Udwadia F.E., Raju N., Some global properties of a pair of coupled maps: quasi-symmetry,
    periodicity and synchronicity, {\it Phys. D} {\bf 111}(1998), 16-26.

    \bibitem{ulam} Ulam S.M., \textit{Problems in modern mathematics},   Wiley, New York, 1964.

    \bibitem{ulam1} Ulam S.M., \textit{A collection of mathematical problems},  Interscience Publisher, New York, 1960.

	\bibitem{28} Vallander S.S., On the limit behavior of iteration sequence of certain quadratic transformations.\textit{ Soviet Math. Doklady}, {\bf 13}(1972), 123--126.
	
	\bibitem{29} Volterra V., Lois de fluctuation de la population de plusieurs esp`eces coexistant dans le m'eme milieu, A\textit{ssociation Franc. Lyon} 1926 (1927), 96--98 (1926).
	

    \bibitem{zakh} Zakharevich M.I., On behavior of trajectories and the ergodic hypothesis
    for quadratic transformations of the simplex, {\it Russian Math. Surveys} {\bf33}1978, 265--266

    \end{thebibliography}
\end{document}